\documentclass[11pt]{amsart}
\usepackage{graphics}
\usepackage{latexsym}
\usepackage{amsmath}
\usepackage{adjustbox}
\usepackage{amssymb,amsthm,amsfonts}
\usepackage{amscd}
\usepackage{syntonly}
\usepackage[arrow, matrix, curve]{xy}
\title[Hilbert modular varieties]
{On the cohomology groups of local systems over Hilbert modular
varieties via Higgs bundles}
\author[Stefan M\"{u}ller-Stach]{Stefan M\"{u}ller-Stach}
\address{Universit\"at Mainz,
Fachbereich 17, Mathematik, 55099 Mainz, Germany}
\email{stach@uni-mainz.de}
\author[Mao Sheng]{Mao Sheng}
\address{School of Mathematical Sciences,
University of Science and Technology of China, Hefei, 230026, China}
\email{msheng@ustc.edu.cn}
\author[Xuanming Ye]{Xuanming Ye}
\address{School of Mathematics and Computational Science,
Sun Yat-sen University, 510275 Guangzhou, P.R.
China}
\email{yexm3@mail.sysu.edu.cn}
\author[Kang Zuo]{Kang Zuo}
\address{Universit\"at Mainz,
Fachbereich 17, Mathematik, 55099 Mainz,
Germany}\email{zuok@uni-mainz.de}
\begin{document}
\theoremstyle{plain}
\newtheorem{thm}{Theorem}[section]
\newtheorem{theorem}[thm]{Theorem}
\newtheorem{lemma}[thm]{Lemma}
\newtheorem{corollary}[thm]{Corollary}
\newtheorem{proposition}[thm]{Proposition}
\newtheorem{addendum}[thm]{Addendum}
\newtheorem{variant}[thm]{Variant}
\theoremstyle{definition}
\newtheorem{construction}[thm]{Construction}
\newtheorem{notations}[thm]{Notations}
\newtheorem{question}[thm]{Question}
\newtheorem{problem}[thm]{Problem}
\newtheorem{remark}[thm]{Remark}
\newtheorem{remarks}[thm]{Remarks}
\newtheorem{definition}[thm]{Definition}
\newtheorem{claim}[thm]{Claim}
\newtheorem{assumption}[thm]{Assumption}
\newtheorem{assumptions}[thm]{Assumptions}
\newtheorem{properties}[thm]{Properties}
\newtheorem{example}[thm]{Example}
\newtheorem{conjecture}[thm]{Conjecture}
\numberwithin{equation}{thm}

\newcommand {\dt} {\bullet}
\newcommand{\sA}{{\mathcal A}}
\newcommand{\sB}{{\mathcal B}}
\newcommand{\sC}{{\mathcal C}}
\newcommand{\sD}{{\mathcal D}}
\newcommand{\sE}{{\mathcal E}}
\newcommand{\sF}{{\mathcal F}}
\newcommand{\sG}{{\mathcal G}}
\newcommand{\sH}{{\mathcal H}}
\newcommand{\sI}{{\mathcal I}}
\newcommand{\sJ}{{\mathcal J}}
\newcommand{\sK}{{\mathcal K}}
\newcommand{\sL}{{\mathcal L}}
\newcommand{\sM}{{\mathcal M}}
\newcommand{\sN}{{\mathcal N}}
\newcommand{\sO}{{\mathcal O}}
\newcommand{\sP}{{\mathcal P}}
\newcommand{\sQ}{{\mathcal Q}}
\newcommand{\sR}{{\mathcal R}}
\newcommand{\sS}{{\mathcal S}}
\newcommand{\sT}{{\mathcal T}}
\newcommand{\sU}{{\mathcal U}}
\newcommand{\sV}{{\mathcal V}}
\newcommand{\sW}{{\mathcal W}}
\newcommand{\sX}{{\mathcal X}}
\newcommand{\sY}{{\mathcal Y}}
\newcommand{\sZ}{{\mathcal Z}}
\newcommand{\A}{{\mathbb A}}
\newcommand{\B}{{\mathbb B}}
\newcommand{\C}{{\mathbb C}}
\newcommand{\D}{{\mathbb D}}
\newcommand{\E}{{\mathbb E}}
\newcommand{\F}{{\mathbb F}}
\newcommand{\G}{{\mathbb G}}
\newcommand{\HH}{{\mathbb H}}
\newcommand{\I}{{\mathbb I}}
\newcommand{\J}{{\mathbb J}}
\renewcommand{\L}{{\mathbb L}}
\newcommand{\M}{{\mathbb M}}
\newcommand{\N}{{\mathbb N}}
\renewcommand{\P}{{\mathbb P}}
\newcommand{\Q}{{\mathbb Q}}
\newcommand{\R}{{\mathbb R}}
\newcommand{\SSS}{{\mathbb S}}
\newcommand{\T}{{\mathbb T}}
\newcommand{\U}{{\mathbb U}}
\newcommand{\V}{{\mathbb V}}
\newcommand{\W}{{\mathbb W}}
\newcommand{\X}{{\mathbb X}}
\newcommand{\Y}{{\mathbb Y}}
\newcommand{\Z}{{\mathbb Z}}
\newcommand{\id}{{\rm id}}
\newcommand{\rank}{{\rm rank}}
\newcommand{\END}{{\mathbb E}{\rm nd}}
\newcommand{\End}{{\rm End}}
\newcommand{\Eis}{{\rm Eis}}
\newcommand{\Hg}{{\rm Hg}}
\newcommand{\tr}{{\rm tr}}
\newcommand{\Sl}{{\rm Sl}}
\newcommand{\Gl}{{\rm Gl}}
\newcommand{\Gr}{{\rm Gr}}
\newcommand{\Cor}{{\rm Cor}}
\newcommand{\Hom}{{\rm Hom}}
\maketitle

\footnotetext[1]{This work was supported by SFB/Transregio 45
Periods, Moduli Spaces and Arithmetic of Algebraic Varieties of the
DFG (Deutsche Forschungsgemeinschaft); The second named author was also supported by National Natural Science Foundation of China (Grant No. 11471298); The third named author was supported by
National Natural Science Foundation of China (Grant No. 11201491)}

\begin{abstract}
Let $X$ be a Hilbert modular variety and $\V$ a non-trivial local
system over $X$ with infinite monodromy. In this paper we study
Saito's mixed Hodge structure (MHS) on the cohomology group
$H^k(X,\V)$ using the method of Higgs bundles. Among other results
we prove the Eichler-Shimura isomorphism, give a dimension formula
for the Hodge numbers and show that the mixed Hodge structure is
split over $\R$. These results are analogous to \cite{MS} in the
cocompact case and complement the results in \cite{Fr} for constant
coefficients.
\end{abstract}

\section{Introduction}\label{introduction}
Consider the Lie group $G=SL(2,\R)^{n}\times U$, where $U$ is
connected and compact. In their classical work \cite{MS} Matsushima
and Shimura study the cohomology groups $H^*(X,\V)$ with values in a
local system $\V$ attached to a linear representation of $G$ on a
compact quotient $X$ of a product of upper half planes by a discrete
subgroup $\Gamma \subset G$, see \S1 in \cite{MS}. The main result
of \cite{MS} is a dimension formula for the Hodge numbers of the
pure Hodge structure on $H^*(X,\V)$.

The arguments and results of the present paper grew out of an
attempt to find a generalization of the results in \cite{MS}. The
use of the maximum principle in the proof of the vanishing result of
Theorem 3.1 in \cite{MS} presents an obvious difficulty for a direct
generalization to the non-compact case. The case of Hilbert modular
surfaces was already studied in the thesis \cite{Sheng} of the
second named author. The final approach we have taken is a mixture
of the one of Zucker in \cite{Zuc1} for general locally symmetric
varieties and the original theory of harmonic forms in \cite{MS} for
discrete quotients of products of upper half planes. The technique
of Higgs bundles made it possible to combine both methods
effectively. The vanishing theorem of Mok on locally homogenous
vector bundles \cite{Mo} in the case of Hilbert modular varieties is
indispensable to obtain our results for \emph{all} non-trivial local
systems with infinite monodromy groups, i.e., also non-regular ones.
We are aware of the fact that some of our results can be also
explained in an automorphic setting. Harris and Zucker
\cite{HZ1,HZ2,HZ3} have developed a general framework using
automorphic forms which is related to our work through the
BGG-complex. However our approach is purely Hodge theoretic and can
be applied to other Shimura varieties and to more general
non-locally-homogenous situations. For example, we have started to
apply the same machinery to orthogonal and unitary Shimura varieties
(see \cite{MMWYZ},\cite{MYZ},\cite{DPMVZ}).

Now let $X$ be a Hilbert modular variety and $\V$ an irreducible local system over $X$.
Then the cohomology group $H^k(X,\V)$ carries a natural real MHS, which is the principal object
of study in the present paper. The case of a constant local system
has been treated in the book \cite{Fr} by Freitag. The relevant result is Theorem 7.9  Ch.
III in ibid. where a formula for the Hodge numbers of the real
mixed Hodge structures $H^k(X,\R), 0\leq k\leq 2n$ is provided. Using with Proposition \ref{split MHS} of the current
paper and Propositions 7.2, 7.7 in ibid., one can even show that the real MHS on the cohomology group is actually split. Note that when $\V$ has finite
monodromy, it becomes trivial after a finite \'{e}tale base change. Thus we
shall assume throughout the paper that $\V$ is a non-trivial local
system with infinite monodromy. Any such irreducible $\V$ is induced by an irreducible linear
representation of $G$ and is of the form $\V_m$ for a certain
$n$-tuple $m=(m_1,\ldots,m_n)\in \N_{0}^{n}$. For such an $m$, one puts
$|m|=\sum_{i=1}^{n}m_i$, and for any subset $I\subset
\{1,\cdots,n\}$, $|m_I|=\sum_{i\in I}m_i$. Each $\V=\V_m$ underlies a natural real polarized variation of Hodge structure ($\R$-PVHS) $\V_{\R}$
of weight $|m|$. After the work of Deligne, Saito and Zucker \cite{De1,Sa1,Sa2,Sa3,Sa4,Zuc}, the cohomology group $H^k(X,\V_{\R})$ carries a natural
real MHS with weights $\geq |m|+k$ (Saito's MHS).
This MHS is defined over $\Q$ if all $m_i$ are equal. See \S\ref{preliminaries} for details. Our main results are summarized as follows:
\begin{theorem}\label{thm1}
Let $X$ be a Hilbert modular variety of dimension $n\geq 2$ and
$\V_m$ be the irreducible non-trivial local system determined by
$m=(m_1,\ldots,m_n)\in \N_{0}^{n}$. Then:
\begin{itemize}
\item[(i)] $H^{k}(X,\V_m)=0$ for $0\leq k\leq n-1$ and $k=2n$.
\item[(ii)] If $m_1=\cdots=m_n$, then for $n+1\leq k\leq 2n-1$
$$
h_k^{|m|+n,|m|+n}:=\dim_\C
F^{|m|+n}W_{2(|m|+n)}H^{k}(X,\V_m)=\dim_\C
H^{k}(X,\V_m)=\binom{n-1}{k-n}h,
$$
where $h$ is the number of cusps.
\item[(iii)] If not all $m_i$ are equal, then $H^{k}(X,\V_m)=0$ for $n+1\leq k\leq 2n-1$.
\item[(iv)] One has the formula
$$
\dim H^{n}(X,\V_m)=\delta(m)h+\sum_{I\subset \{1,\cdots,n\}}h(I,m).
$$
where $\delta(m)=1$ if $m_1=\cdots=m_n$ is satisfied and otherwise
zero, and $h(I,m)$ is the dimension of cusp forms on $\HH^n$ with
respect to $\Gamma_I$ (see \S\ref{intersection cohomology} and
\S\ref{final results} for more details). Moreover, for $P+Q=|m|+n$,
$$
h^{P,Q}_{n}:=\dim_\C Gr^{P}_{F}Gr^{Q}_{\bar
F}Gr^{W}_{|m|+n}H^n(X,\V_m)=\sum_{\stackrel{I\subset
\{1,\ldots,n\},}{ |m_I|+|I|=P}}h(I,m),
$$
and $h_n^{|m|+n,|m|+n}=\delta(m)h$. Otherwise $h_n^{P,Q}=0$.
\end{itemize}
\end{theorem}
For the convenience of the reader we give a short account on a
conjecture of Harris-Zucker \cite{HZ3}. Recall that the Zucker conjecture, which has been solved by Looijenga and Saper-Stern
via different methods, asserts the following statement.

\begin{theorem}[Looijenga \cite{Lo}, Saper-Stern \cite{SS}]\label{Zucker conjecture}
Let $M$ be a smooth arithmetic quotient of a Hermitian symmetric
domain with $M^*$ the Baily-Borel compactification. Let $\W$ be a
locally homogenous VHS over $M$. Let $g$ (resp. $h$) be a group
invariant metric on $M$ (resp. $\W$), and
let $H^k_{(2)}(M,\W)$ be the $L^2$-cohomology group of degree $k$
with coefficients in $\W$ with respect to the above metrics. Let
$IH^k(M^*, \W)$ be the $k$-th (middle perversity) intersection
cohomology. Then one has a natural isomorphism
$$
r_k: H^k_{(2)}(M,\W)\cong IH^k(M^*,\W).
$$
\end{theorem}
The case of Hilbert modular varieties had been known to Zucker
\cite{Zuc2} before the conjecture was proven in the general case. If $\W$ is assumed to be an $\R$-VHS,
then the $L^2$-cohomology $H^k_{(2)}(M,\W)$ carries a real Hodge structure by the theory of $L^2$-harmonic
forms and the intersection cohomology $IH^k(M^*,\W)$ a real Hodge structure by Saito's theory on mixed Hodge modules.
A priori, these two Hodge structures do not coincide under the isomorphism in the above theorem. The following conjecture remains open in general:
\begin{conjecture}[Harris-Zucker, Conjecture 5.3 \cite{HZ3}]\label{HZ conjecture}
Assume $\W$ to be an $\R$-VHS. Then the isomorphism in Theorem \ref{Zucker conjecture} is an isomorphism of real Hodge structures.
\end{conjecture}
Zucker showed some instances for the above conjecture, including the case of Hilbert modular surface with constant coefficients \cite{Zuc3}.
As a byproduct of our study of Saito's MHS on $H^k(X,\V_\R)$, we obtain the following case of the conjecture:
\begin{theorem}\label{thm4}
Let $X$ be a Hilbert modular variety and $X^*$ its Baily-Borel compactification. Let $\V_\R$ be an irreducible $\R$-VHS with infinite monodromy.
Then the natural isomorphism $r_k: H_{(2)}^{k}(X^*,\V_\R)\cong IH^{k}(X^*,\V_\R) $ is an isomorphism of real Hodge structures.
\end{theorem}
Henceforth, we identify these two Hodge structures under the natural isomorphism in Theorem \ref{Zucker conjecture}. Our next result is to show that
Saito's MHS is split over $\R$. Let $j: X\hookrightarrow X^*$ be the natural inclusion. Then it induces an injective morphism of MHS
$IH^k(X^*,\V_{\R})\cong H_{(2)}^k(X^*,\V_{\R})\to
H^{k}(X,\V_{\R})$ (see Proposition \ref{injectivity}). We denote again by $IH^k(X^*,\V_{\R})$ the image of the
embedding in the following. The theory of Eisenstein cohomology (see \cite{Ha2},
\cite{Schw}) provides a decomposition $H^{k}(X,\V_\R)=H^{k}_{!}(X,\V_{\R})\oplus H_{\Eis}^{k}(X,\V_\R)$,
where $H^{k}_{!}(X,\V_{\R})$ is the image of $H_{c}^{k}(X,\V_\R)$,
the cohomology of $\V_{\R}$ with compact supports, in
$H^{k}(X,\V_\R)$.
\begin{theorem}\label{thm2}
For $n\leq k\leq 2n-1$ let $(H^k(X,\V_\R),W_.,F^.)$ be Saito's MHS.
Then:
\begin{itemize}
\item [(i)] For $n+1\leq k\leq 2n-1$, one has $H^k(X,\V_\R)=H^k_{\Eis}(X,\V_\R)$ and the MHS on
$H^k(X,\V_\R)$ is pure Hodge-Tate of type $(|m|+n,|m|+n)$.
\item [(ii)] $IH^n(X^*,\V_{\R})=H^{n}_{!}(X,\V_{\R})$ and
$H^n(X,\V_\R)=IH^n(X^*,\V_{\R})\oplus H^n_{\Eis}(X,\V_\R)$ is a
natural splitting of MHS over $\R$ into two pieces with weights
$|m|+n$ and $2(|m|+n)$. 
\end{itemize}
\end{theorem}

Due to the splitting of the MHS $(H^k(X,\V_\R),W_.,F^.)$ over $\R$,
one has a bigrading $H^k(X,\V_m)=\bigoplus_{P,Q}H^{P,Q}_{k}$ with
$H^{P,Q}_{k}=F^P\cap \bar F^{Q}\cap W_{P+Q,\C}$ satisfying
$$
W_{l,\C}=\bigoplus_{P+Q\leq l}H^{P,Q}_k,\ F^{P}=\bigoplus_{r\geq
P}H^{r,s}_k.
$$
In addition, we can give an algebraic description of the (non-zero)
Hodge $(P,Q)$-components $H^{P,Q}_{k}$. This result can be regarded
as a generalized \emph{Eichler-Shimura isomorphism} for
Hilbert modular varieties:
\begin{theorem}\label{thm3}
Let $\bar X$ be a smooth toroidal compactification of $X$ (see
\cite{AMRT}) with boundary divisor $S=\bar X-X$. For $1\leq i\leq n$
let $\sL_i$ be the good extension (see \cite{Mum}) to $\bar X$ of
the locally homogenous line bundle over $X$ determined by the
automorphy factor $c_iz_i+d_i$ (using the notation of
\S\ref{preliminaries}). Then one has the following natural
isomorphisms:
\begin{itemize}
\item [(i)] For $n+1\leq k\leq 2n-1$,
$H^{k}(X,\V_m)=H_k^{|m|+n,|m|+n}\cong H^{k-n}(\bar X,
\bigotimes_{i=1}^{n}\sL^{m_i+2}_{i})$.
\item [(ii)] $H_n^{|m|+n,0}\cong H^{0}(\bar X, \sO_{\bar X}(-S)\otimes
\bigotimes_{i=1}^{n}\sL^{m_i+2}_{i}),\ H_n^{|m|+n,|m|+n}\cong
H^{0}(S, \bigotimes_{i=1}^{n}\sL^{m_i+2}_{i}|_S)$, and for $0\leq
P\leq |m|+n-1, \ P+Q=|m|+n$,
$$
H_n^{P,Q}\cong \bigoplus_{\stackrel{I\subset \{1,\ldots,n\},}{
|m_I|+|I|=P}}H^{k-|I|}(\bar X, \bigotimes_{i\in
I}\sL^{m_i+2}_{i}\otimes \bigotimes_{i\in I^c}\sL^{-m_i}_{i}).
$$
\end{itemize}
\end{theorem}

In the above results, Theorem \ref{thm1}, (i) and the first half of
Theorem \ref{thm2}, (i) for regular local systems are special cases
of Li-Schwermer \cite{LS} (see also Saper \cite{Sap}). Wildeshaus
has recently informed us that the main result in \cite{BW} also
implies Theorem \ref{thm1}, (ii). Combined with Lemma \ref{lemma on
MHS}, one is able to show then the second half of Theorem
\ref{thm2}, (i) for regular local systems.

The paper is organized as follows: Section \ref{preliminaries}
contains the basic set-up. In section \ref{logarithmic higgs
cohomology} we compute the logarithmic Higgs cohomology and present
an algebraic description of the gradings of the Hodge filtration on
cohomology. Section \ref{intersection cohomology} studies the pure
Hodge structure on the $L^2$-cohomology which relies on the theory
of $L^2$-harmonic forms. This provides an $L^2$-generalization of
the results in \cite{MS}. Section \ref{Eisenstein cohomology}
introduces Eisenstein cohomology. This has been extensively
investigated in \cite{Ha1}, \cite{Fr} for constant coefficients and
in \cite{Zuc2, HZ1,HZ2,HZ3}, \cite{LS} for non-constant
coefficients. We complement these results by computing the Hodge
type of Eisenstein cohomology. The last section combines all results
and contains the proofs of the above theorems.

\textbf{Acknowledgements.} We would like to thank Michael Harris and
J\"{o}rg Wildeshaus for their interest and comments on this paper.
The referee has made a tremendous effort to make the presentation as
clear as possible. His/Her advice about Hodge structures on
$L^2$-cohomology and intersection cohomology greatly improved
this paper. In particular, Theorem \ref{thm4} came out during the revision of the paper.
We thank him/her heartily.

\section{Preliminaries and Saito's mixed Hodge structure}\label{preliminaries}

Let $F\subset \R$ be a totally real number field of degree $n \geq
2$ over $\Q$, with the set of real embeddings ${\rm
Hom}_{\Q}(F,\R)=\{\sigma_1=id,\ldots,\sigma_n\}$. Let $\sO_F$ be the
integer ring of $F$ and $\sO_F^*$ be the unit group of $\sO_F$. For
an element $a\in F$ one puts $a^{(i)}=\sigma_i(a)$, the $i$-th
Galois conjugate of $a$. Let $G=R_{F|\Q}SL_2$ be the $\Q$-algebraic
group obtained by Weil restriction. The set of real points $G(\R)$
of $G$ is identified with $G_1\times \cdots \times G_n$, where each
$G_i$ is a copy of $SL(2,\R)$. The subset $G(\Q)\subset G(\R)$ is
then given by
$$
\Biggl\{ \left( \left(
     \begin{array}{cc}
       a^{(1)} & b^{(1)} \\
       c^{(1)} & d^{(1)} \\
     \end{array}
   \right)
,\ldots,\left(
     \begin{array}{cc}
       a^{(n)} & b^{(n)} \\
       c^{(n)} & d^{(n)} \\
     \end{array}
   \right) \right) \in G_1\times \cdots\times G_n|\ a,b,c,d\in F \Biggr\}.
$$
Now let $\HH^{n}=\HH_1\times \cdots\times \HH_n$ be the product of
$n$ copies of the upper half plane with coordinates
$$
z=(z_1=x_1+iy_1,\ldots,z_l=x_l+iy_l,\ldots,z_n=x_n+iy_n).
$$
The group $G(\R)$ acts on $\HH^n$ by a product of linear fractional
transformation. Namely, for $g=\Biggl( g_1=\left(
\begin{array}{cc}
a_{1} & b_1 \\
c_1 & d_1 \\
\end{array}
\right), \ldots, g_n=\left(
\begin{array}{cc}
a_n & b_n \\
c_n & d_n \\
\end{array}
\right) \Biggr) \in G(\R)$ and $z\in \HH^n$, the action is given by
$g\cdot z=(g_1\cdot z_1,\ldots,g_n\cdot z_n)$, where $g_i\cdot
z_i=(a_iz_i+b_i)(c_iz_i+d_i)^{-1}$. The action is transitive and the
isotropy subgroup of $G(\R)$ at the base point $z_0=(i,\ldots,i)$ is
$K=SO(2)^{n}$, a maximally compact subgroup. $K$ acts on $G$ by
right multiplication and one identifies the set $G/K$ of left cosets
naturally with $\HH^n$. Let $\Gamma\subset G(\Q)$ be a torsion-free
subgroup which is commensurable with $G(\Z)\subset G(\Q)$. It is
called a \emph{Hilbert modular group} in this paper and will be
fixed throughout. By the theorem of Baily-Borel, the quotient space
$X_\Gamma:=\Gamma\backslash G(\R)/K$ is naturally a smooth
quasi-projective variety, which is called the \emph{Hilbert modular
variety} for $\Gamma$. As $\Gamma$ is fixed, we shall denote
$X_\Gamma$ simply by $X$. The topological space $X$ is non-compact
and admits several natural compactifications. The Baily-Borel
compactification $X^*$ of $X$ is obtained by adding a number $h$ of
\emph{cusps} as a set (see \cite{Fr}) and has the structure of a
projective variety by Baily-Borel. It is however singular, and
admits a natural family of resolutions of singularities, the
so-called smooth toroidal compactifications (see \cite{AMRT} for
general locally symmetric varieties and \cite{Eh} more details for
Hilbert modular varieties). Let $\bar X$ be such a smooth
compactification, and let $S=\bar X-X$ be the divisor at infinity,
which has simple normal crossings. In addition, we also use the
Borel-Serre compactication $X^\sharp$ in \S\ref{Eisenstein
cohomology}. It is a smooth compact manifold with boundary, which
contains $X$ as the interior open subset. The boundary $\partial
X^\sharp=X^\sharp-X$ has $h$ components in total, with each
component an $(S^1)^{n}$-bundle over $(S^1)^{n-1}$.

For $A\subset \C$ a $\Q$-subalgebra, we define an $A$-local system
over $X$ as a locally constant sheaf of free $A$-modules of finite
rank with respect to the analytic topology on $X$. Let $0$ be the
point of $X$ given by the $\Gamma$-equivalence class of $z_0\in
\HH^{n}$. Then it is well-known that an $A$-local system over $X$
corresponds to a representation $\pi_1(X,0)\to GL(A)$. In this paper
$A$ is either $\Q$, $\R$ or $\C$. Since $\pi_1(X,0)$ is naturally
identified with $\Gamma$, by the super-rigidity theorem of Margulis
(see \cite{Ma}), equivalence classes of complex local systems with
infinite monodromy groups over $X$ are in one-to-one correspondence
with equivalence classes of finite dimensional complex
representations of $G(\R)$. Let $\V$ be the complex local system
corresponding to the irreducible representation $\rho: G(\R)\to
GL(V)$. By Schur's lemma, there exists an $n$-tuple
$m=(m_1,\ldots,m_n)$ of non-negative integers and $n$ copies $V_i$
of $\C^2$ for $i=1,\ldots,n$ such that $\rho=\rho_{m_1}\otimes
\cdots\otimes \rho_{m_n}$, where for each $i$, $\rho_{m_i}: G_i\to
GL(S^{m_i}V_i)$ is isomorphic to the $m_i$-th symmetric power of the
standard complex representation of $SL(2,\R)$. A local system $\V$
is called \emph{regular} if $\rho$ is a regular representation,
i.e., its highest weight is contained in the interior of the Weyl
chamber. It is clear that $\V=\V_m$ is regular if and only if each
$m_i$ in the above is positive. To summarize, for each $m\in
\N_{0}^{n}$, there is a unique complex local system $\V_m$ over $X$
up to isomorphism, and any complex local system over $X$ is a finite
direct sum of such. For $(0,\ldots,m_i,\ldots,0)$ we denote the
corresponding local system by $\V_{i,m_i}$. So
$\V_m=\V_{1,m_1}\otimes \cdots\otimes \V_{n,m_n}$. The complex local
system $\V=\V_m$ is the complexification of a natural real local
system $\V_{\R}$. Moreover, $\V_{\R}$ is naturally an $\R$-PVHS,
which is a special case considered by Zucker (see \cite{Zuc1}). This
can be seen as follows. Let $ e_1=\left(\begin{array}{c}
1 \\
0 \\
\end{array} \right),
\ e_2=\left( \begin{array}{c}
0 \\
1 \\
\end{array} \right)
$ be the standard basis of $\C^2$ on which $SL(2,\R)$ acts by matrix
multiplication. Then $\R^2=\R e_1 + \R e_2 \subset \C^2$ is an
invariant $\R$-structure. Define a symplectic form $\omega$ on
$\R^2$ such that $\{e_1,e_2\}$ is the symplectic basis for $\omega$.
Let $\C^{1,0}=\C\{e_1-ie_2\}$ and $\C^{0,1}=\C\{e_1+ie_2\}$. Then
the decomposition $\R^2\otimes_{\R}\C=\C^{1,0}\oplus \C^{0,1}$
defines a polarized weight one Hodge structure on $\R^2$, and this
decomposition has the special property that it is also the
eigen-decomposition of the induced action of $SO(2)$ by restriction.
For each $1\leq i\leq n$, the $i$-th factor of $G(\R)$ acts on $V_i$
via the standard representation and trivially on any other factor.
Applying the foregoing construction on $\C^2$ to $V_i$, one obtains
a polarized weight one Hodge structure on the  fiber of the constant
bundle $\HH^n\times V_i$ at $z_0\in \HH^n$. By using the homogeneity
property of $\HH^n$, one defines a $\R$-PVHS on the constant bundle
$\HH^n\times V_i$, and it descends to a $\R$-PVHS on $\V_{i,1}$ over
$X$ (see \S4 in \cite{Zuc1}). Taking the $m_i$-th symmetric power,
one obtains a $\R$-PVHS of weight $m_i$ on $\V_{i,m_i}$, and further
by taking tensor products one obtains a $\R$-PVHS of weight
$|m|=\sum_{i=1}^{n}m_i$ on $\V_m$ as claimed. It is clear that
$\V_{\R}$ is in fact defined over $F\subset \R$, and is even defined
over $\Q$ if (and only if) $m_1=\cdots=m_n$ holds.

Now we consider an even more general setting. Let $M$ be a
quasi-projective manifold of dimension $d$ and
$(\W_{\R},\nabla,F^\cdot)$ a $\R$-PVHS over $M$ of weight $n$. Let
$\bar M$ be a smooth, projective compactification of $M$ such that
$S=\bar M-M$ is a simple normal crossing divisor. For simplicity of
exposition, we assume that the local monodromy around each
irreducible component of $S$ is unipotent. Put $\W_{an}=\W_{\R}\otimes_{\R}\sO_{M_{an}}$, where
$\sO_{M_{an}}$ is the sheaf of germs of holomorphic functions on
$M$. Deligne's canonical extension (see Ch II, \S4 in \cite{De0})
gives a unique extended vector bundle $\bar \W_{an}$ of $\W_{an}$
over $\bar M$, together with a flat logarithmic connection $\bar
\nabla: \bar \W_{an}\to \bar \W_{an}\otimes \Omega^1_{\bar
M_{an}}(\log S)$. Using this we obtain the logarithmic de Rham
complex $\Omega^*_{\log}(\bar \W_{an},\bar \nabla)$. By Deligne in
loc. cit., the complex computes the cohomology groups
$H^{\cdot}(M,\W_{\R})$. Schmid's nilpotent orbit theorem implies
that the Hodge filtration $F^\cdot$ extends to a filtration $\bar
F^\cdot$ of holomorphic subbundles (not merely subsheaves) of $\bar \W_{an}$ as well (see
\S4 in \cite{Sch}). By GAGA, the extended holomorphic objects over
$\bar M$ are in fact algebraic. One defines a Hodge filtration on
the logarithmic de Rham complex by $ F^r\Omega^*_{\log}(\bar
\W_{an},\bar \nabla)=\Omega^*_{\bar M}(\log S)\otimes \bar F^{r-*}$,
which is a subcomplex by Griffiths transversality. After Saito (see
\cite{Sa1}-\cite{Sa4}), there is a naturally defined weight
filtration $W_\cdot$ on the logarithmic de Rham complex.

For convenience,
we shall briefly recall the construction of Saito's MHS on $H^*(M,\W_{\R})$ (all the details can be found in \cite{Ar}).

Extend $(\W_{an},\nabla)$ to a vector bundle $\bar \W_{an}^I$ with a logarithmic connection such that the eigenvalues of its
residues lie in a half open interval $I$ of length one. Note that $\bar \W_{an}
=\bar \W_{an}
^{[0,1)}$
and define $\sM: = \bigcup_I \bar \W_{an}
^I\subseteq j_*\W_{an}
$.
Then $\sM$ is a $D_M$-module which corresponds to the perverse sheaf
$\R j_* \W_{\R}[n]\otimes \C$. Filter this by
$$F_p\sM =\sum_i F_iD_MF_{p-i}\bar \W_{an}^{(-1,0]}$$
Then according to \cite{Sa2}:
\begin{thm}
There exist compatible filtrations $W$ on $\R j_* \W_{\R}[n]$ and $\sM$ such that
$(\R j_* \W_{\R}[n],   W, \sM, F)$ becomes a mixed Hodge module.
\end{thm}
Then one deduces Saito's  MHS on $H^*(M,\W_{\R})$
by elaborating on a remark in \cite{Sa2}. There is a natural isomorphism in the derived category of sheaves of abelian groups:
$$\alpha:\R j_* \W_{\R}\otimes \C\cong \Omega_{\log}^*(\bar \W_{an},\bar\nabla).$$
Using the isomorphism, the filtration on the left is transported into the Hodge filtration on the logarithmic de Rham complex described as above.
There is also the induced weight filtration $W_{\cdot}$ on the complex. One constructs $W_{\cdot}$ by taking the filtration induced by
$(\Omega_M^*\otimes W_{*+n} \sM)[n]$
under the inclusion
$$\Omega_{\bar M}^*(\log S)\otimes \bar \W_{an}\subset
\Omega_M^*\otimes \sM[n].$$
We recall at this point that Deligne introduced a device, called a {\em cohomological mixed Hodge complex} \cite[\S 8.1]{De1}, for
producing mixed Hodge structures. This consists of
a bifiltered complex $(A_\C,W_\C, F^{\cdot})$ of sheaves of  $\C$-vector spaces,
a filtered complex of $(A,W)$ of sheaves over $\R$, and
a filtered quasi-isomorphism $(A,W)\otimes \C\cong (A_\C,W_\C)$.
The crucial axioms are that
\begin{itemize}
\item  this datum should induce a pure weight $i+k$ real Hodge
  structure on $H^i(Gr^W_kA)$;
\item the filtration induced by $F^{\cdot}$ on $Gr^W_k$ is strict, i.e.,
the map $H^i(F^pGr^W_kA)\to H^i(Gr^W_kA)$ is injective.
\end{itemize}

\begin{proposition} \label{MHC} With these filtrations
  $$(\R j_* \W_{\R}, W_{\cdot}; \Omega_{\bar M}(\log S)^\dt\otimes  \bar \W_{an},  W_{\cdot}, F^{\cdot};\alpha)$$
 becomes a cohomological mixed Hodge complex.
\end{proposition}

To verify the above axioms, one observes that $Gr^{ W}_k\R j_*\W_{\R}$
decomposes into a direct sum of intersection cohomology complexes associated to pure variations of Hodge structure of the correct weight. An appeal to
theorem of \cite{CKS2} or \cite{KK}  shows that $H^i(Gr_k^{ W}\R j_*\W_{\R})$ carries a pure Hodge structure of weight  $k+i$. Strictness  follows from
similar considerations.

From \cite[\S 8.1]{De1}, we obtain Saito's MHS on $H^k(M,\W_{\R})$:

\begin{corollary}
Notation as above. Then the cohomology group $H^k(M,\W_{\R})$ admits a real mixed Hodge structure of weight $n+k$.
Namely, the Hodge filtration is induced by $F^{\cdot}$ under the isomorphism
$$H^k(M ,  \W_{\R} \otimes \C)\cong \mathbb{H}^k(\Omega_{\log}^*(\bar \W_{an},\bar \nabla)),$$
and the weight filtration is determined by
$$W_{l+k} H^k(M,\W_{\R}  \otimes \C): = \mathbb{H}^k(X, W_l(\Omega_{\log}^*(\bar \W_{an},\bar \nabla)).$$
\end{corollary}

Finally, we remark that in the simplest case $\W_{\R}=\R$, $W_{\cdot}$ above
coincides with the filtration $W$ defined by Deligne \cite[\S 3.1]{De1}. In
particular, Saito's MHS coincides with Deligne's in this case.

It is this MHS that we intend to understand properly in the case of Hilbert modular varieties.

After taking the graded object associated to the pair $(\bar
\W_{an},\bar \nabla)$ with respect to the filtration $\bar F^\cdot$,
one obtains the logarithmic Higgs bundle
$(E=\bigoplus_{p+q=n}E^{p,q},\theta=\bigoplus_{p+q=n}
\theta^{p,q})$. Precisely, $E^{p,q}= \bar F^p /\bar F^{p+1} $ and
$\theta^{p,q}=Gr^p_{\bar F^\cdot}\bar \nabla: E^{p,q}\to
E^{p-1,q+1}\otimes \Omega^1_{\bar M}(\log S)$ which is $\sO_{\bar
M}$-linear. The Higgs field is integrable, i.e., $\theta \wedge
\theta =0$, and so one can form the \emph{logarithmic Higgs complex}
$\Omega^*_{\log}(E,\theta)$, which is the logarithmic version of the
Higgs complex over a compact case. The complex appears on page 24 of
\cite{Sim} under the name holomorphic Dolbeault complex. More
specifically, the $p$-th term of the complex
$\Omega^p_{\log}(E,\theta)=E\otimes \Omega^p_{\bar M}(\log S)$ and
the differential $\Omega^p_{\log}(E,\theta)\to
\Omega^{p+1}_{\log}(E,\theta)$ is the composite
\begin{eqnarray*}
   E\otimes \Omega^p_{\bar M}(\log S)& \stackrel{\theta\otimes id}{\longrightarrow} & [E\otimes \Omega^1_{\bar M}(\log S)]\otimes \Omega^p_{\bar M}(\log
S) \\
    &=&  E\otimes [\Omega^1_{\bar M}(\log S)\otimes\Omega^p_{\bar M}(\log
S)]\stackrel{id\otimes pr}{\longrightarrow}E\otimes
\Omega^{p+1}_{\bar M}(\log S).
\end{eqnarray*}
Note that the sheaves and the differentials in this complex are
algebraic. The hypercohomology of this complex is called
\emph{logarithmic Higgs cohomology}. One observes that the
logarithmic Higgs complex is a direct sum
$\bigoplus_{P=0}^{d+n}\Omega^*_{P}(E,\theta)$ of subcomplexes, where
$\Omega^{l}_{P}(E,\theta)=E^{P-l,n-P+l}\otimes \Omega^{l}_{\bar
M}(\log S)$. We define $\sC^{P,l}(E,\theta)$ to be the $l$-th
cohomology sheaf of the subcomplex $\Omega^{*}_{P}(E,\theta)$. The
relation of logarithmic Higgs cohomology with Saito's MHS is given
by the following

\begin{proposition}\label{logarithmic higgs cohomology computes the grading of hodge
filtration} Let $\W=\W_{\R}\otimes \C$ and $F^\cdot$ be the Hodge
filtration on $H^k(M,\W)$ as above. For $0\leq k\leq d$, one has a
natural isomorphism
$$
Gr_{F^{\cdot}}^{P}H^k(M,\W)\cong \HH^{k}(\bar M, \Omega^*_{P}(E,\theta)).
$$
\end{proposition}
\begin{proof}
Deligne Scholie 8.1.9 (v) \cite{De1} asserts that the Hodge filtration on a cohomological mixed Hodge complex
degenerates at the $E_1$-term. So the statement is just a direct consequence of Proposition \ref{MHC}.
\end{proof}
When $\W$ is constant, the Higgs field is trivial. Therefore the
logarithmic Higgs cohomology is just the usual sheaf cohomology. In
another important case the above result can also be improved:
\begin{proposition}\label{degeneration of ss}
Assume that $M$ is a smooth arithmetic quotient of a Hermitian
symmetric domain and $\bar M$ a smooth toroidal compactification.
Moreover, assume that $\W_{\R}$ is locally homogenous. Then one has a natural isomorphism
$$
Gr_{F}^{P}H^k(M,\W_{\R})\cong \bigoplus_{l=0}^{d}H^{k-l}(\bar M,
\sC^{P,l}(E,\theta)).
$$
\end{proposition}
The proof is based on the one of Proposition 5.19 in \cite{Zuc1} for
compact $M$. The new ingredient is the relation between the Mumford extension
of Higgs complex over $M$ and the logarithmic Higgs complex over $\bar M$.
\begin{lemma}\label{Mumford extension}
Let $\Omega^*(E^0,\theta^0)$ be the Higgs complex of
$\W_{\R}$ over $M$, and $\Omega_{\log}^*(E,\theta)$ the logarithmic
Higgs complex of $\W_{\R}$ over $\bar M$. Then the cohomology sheaf of the logarithmic Higgs complex is the good extension
of the corresponding cohomology sheaf in the sense of Mumford \cite{Mum}.
\end{lemma}
\begin{proof}
The Higgs complex is just the restriction of the logarithmic
Higgs complex to $M$. We consider the terms in the log complex. Each term is a tensor product of a wedge power
of $\Omega^1_{\bar{M}}(\log S)$ and the Deligne-Schmid extension of
$E^0=E|_M$. Proposition 3.4 in \cite{Mum} shows that
$\Omega^1_{\bar{M}}(\log S)$ is the good extension of
$\Omega^1_{M}$. The estimates of the Hodge metric, which is the
invariant metric on $E^0$ up to scalar, provided by Theorem 5.21 in
\cite{CKS}, shows that the sections of $E$ are at most of
logarithmic growth around $S$ with respect to the Hodge metric. By
the characterization of the good extension in Proposition 1.3,
\cite{Mum}, one concludes that $E$ is the good extension of $E^0$.
We consider also the bundle homomorphism: we may choose a
local basis of group invariant sections for $E^0$ and conclude that
$\theta^0$ is a constant morphism in this basis, since, by
homogeneity, $\theta^0$ is determined by its value at one point. In
the extended group invariant basis $\theta$ is also constant on
$\bar M$. Therefore, it follows that the $l$-th cohomology sheaf
$\sC^{P,l}(E,\theta)$ of $\Omega_P^*(E,\theta)$ is the good
extension of $l$-th cohomology of $\Omega_P^*(E^0,\theta^0)$ to
$\bar M$, which is a direct summand of $\Omega_P^l(E,\theta)$.
Moreover, the inclusion
$\bigoplus_{l}\sC^{P,l}(E,\theta)[-l]\hookrightarrow
\Omega_P^*(E,\theta)$ is a quasi-isomorphism.
\end{proof}

\section{Logarithmic Higgs cohomology over a Hilbert modular
variety}\label{logarithmic higgs cohomology}

Let $\V=\V_m$ (resp. $\V_{i,m_i}$) be an irreducible complex local
system over the Hilbert modular variety $X$. After a possible finite
\'{e}tale base change of $X$, the local monodromies of $\V$ at
infinity can be made unipotent, so we assume this from now on. This
assumption can be removed once a result is unaffected by finite
\'{e}tale base change. Let $(E_m,\theta_m)$ (resp.
$(E_{i,m_i},\theta_{i,m_i})$) be the resulting logarithmic Higgs
bundle over $\bar X$ by the construction in \S\ref{preliminaries}.
It is clear that the construction is compatible with direct sums and
tensor products, and hence
$$
(E_m,\theta_m)=(E_{1,m_1},\theta_{1,m_1})\otimes \cdots \otimes
(E_{n,m_n},\theta_{n,m_n}).
$$
In this section the logarithmic Higgs cohomology of $(E_m,\theta_m)$
will be determined.

We start with a basic property of the set
$\{(E_{i,1},\theta_{i,1}),\ 1\leq i\leq n\}$ of logarithmic Higgs
bundles.
\begin{definition}
For $1\leq i\leq n$ define $\sL_i$ to be the good extension of the
locally homogenous line bundle over $X$ corresponding to the
automorphy factor $c_i+d_iz_i$.
\end{definition}
The $\sL_i$'s are locally free, either by the proof of Main Theorem 3.1 \cite{Mum}, or by noticing that
$\sL_i$ is the extended Hodge filtration of $\bar \V_{i,1,an}$ by the nilpotent orbit theorem.
\begin{proposition}\label{basic logarithmic higgs bundles}
For each $i$, $E_{i,1}=E^{1,0}_{i,1}\oplus E^{0,1}_{i,1}$ with
$E^{1,0}_{i,1}=\sL_i$ and $E^{0,1}_{i,1}=\sL_i^{-1}$. There exists a
natural isomorphism $\Omega^1_{\bar X}(\log S)=
\bigoplus_{i=1}^{n}\sL_i^2$ such that $\theta_{i,1}^{1,0}:
\sL_{i}\to \sL_{i}^{-1}\otimes \Omega^1_{\bar X}(\log S)$ is the
composition of the tautological maps:
$$\sL_i\stackrel{\cong}{\longrightarrow} \sL_i^{-1}\otimes
\sL_i^2\hookrightarrow \sL_{i}^{-1}\otimes \Omega^1_{\bar X}(\log
S).$$
\end{proposition}
\begin{proof} This follows from the fact that the period map for a Hilbert modular
variety is an embedding together with uniqueness of the Mumford
extension.
\end{proof}
\begin{proposition}\label{cohomology sheaves matrix}
For any subset $I\subset \{1,\ldots,n\}$ define $I^c$ to be the
complement of $I$, $|m_I|=\sum_{i\in I}m_i$ and
$\sC_{I}=\bigotimes_{i\in I}\sL_i^{m_i+2} \otimes \bigotimes_{i\in
I^{c}} \sL_i^{-m_i}$. Then one has the formula
$$
\sC^{P,l}(E_m,\theta_m)=\bigoplus_{\stackrel{I\subset
\{1,\ldots,n\},}{|m_I|+|I|=P,|I|=l}}\sC_{I}.
$$
\end{proposition}
\begin{proof}
By Lemma \ref{Mumford extension}, each
$\sC^{P,l}(E_m,\theta_m)$ is the good extension of the corresponding
cohomology sheaf for the Higgs complex, the restriction of
$\Omega_P^*(E_m,\theta_m)$ to $X$. By the homogeneity of the Higgs
complex, it suffices to carry out the computation at $0$, or
equivalently the computation of the pull-back Higgs complex over
$\HH^{n}$ at $z_0$. For simplicity, we shall not change the previous
notation when we carry out the calculation over $\HH^n$. First we
observe that over $\HH^{n}$ the Higgs complex decomposes
$$
\Omega^{*}_{P}(E_m,\theta_m)=\bigoplus_{P_1+\cdots+P_n=P}
\Omega^{*}_{P_1}(E_{1,m_1},\theta_{1,m_1})\otimes
\cdots\otimes\Omega^{*}_{P_n}(E_{n,m_n},\theta_{n,m_n}).
$$
Now, by the homogeneity of the Higgs complex, the condition in Appendix 8 in \cite{EV} is satisfied.
In fact, it suffices to check this condition at the point $z_0$.
Therefore, it follows from Appendix 8 in \cite{EV} that each cohomology sheaf of the previous complex decomposes into the following
form:
$$
\sC^{P,l}(E_m,\theta_m)=\bigoplus_{\stackrel{P_1+\cdots
+P_n=P,}{l_1+\cdots+l_n=l}}\sC^{P_1,l_1}(E_{1,m_1},\theta_{1,m_1})\otimes\cdots\otimes
\sC^{P_n,l_n}(E_{n,m_n},\theta_{n,m_n}).
$$
Now we calculate the single factor
$\sC^{P_1,l_1}(E_{1,m_1},\theta_{1,m_1})$: From the Higgs complex
$$
\begin{adjustbox}{scale=0.8}
\xymatrix{
\,
& \,
& \mathcal{L}^{m_1}_1 \ar[d]^{\cong}
& \oplus
& \cdots
& \oplus
& \mathcal{L}^{-m_1+2}_1 \ar[d]^{\cong}
& \oplus
& \mathcal{L}^{-m_1}_1 \ar[d]^{}
\\
\mathcal{L}^{m_1}_1 \otimes \Omega_\mathbb{H}
& \oplus
& \mathcal{L}^{m_1-2}_1 \otimes \Omega_\mathbb{H}
& \oplus
& \cdots
& \oplus
& \mathcal{L}^{-m_1}_1 \otimes \Omega_\mathbb{H}
& \oplus
& 0,
}
\end{adjustbox}
$$

and the isomorphism $\Omega_{\HH}\cong \sL_1^2$, we see that
$$
\sC^{P_1,l_1}(E_{1,m_1},\theta_{1,m_1})=\left\{
                                          \begin{array}{lll}
                                            \sL_1^{-m_1} , & \hbox{if $P_1=0,l_1=0$;} \\
                                             \sL_1^{m_1+2} , & \hbox{if
                                             $P_1=m_1+1,l_1=1$;}\\
                                             0, &\hbox{otherwise.}
                                          \end{array}
                                        \right.
$$
The remaining step is elementary.
\end{proof}
Since $C_I$ is a tensor product of powers of $\sL_i$'s, it is locally free.
The above proposition and Proposition \ref{degeneration of ss} imply
the following
\begin{corollary}\label{algebraic description}
Let $0\leq k\leq 2n$ and $F^.$ be the Hodge filtration on
$H^k(X,\V_m)$. For each $0\leq P\leq |m|+k$, one has a natural
isomorphism
$$
Gr_{F}^{P}H^k(X,\V_m)\cong \bigoplus_{\stackrel{I\subset
\{1,\ldots,n\},}{|m_I|+|I|=P}}H^{k-|I|}(\bar X,\sC_I).
$$
\end{corollary}

\section{A vanishing theorem of Mok}
In this section we deduce a certain vanishing result of global
sections of $\sC_I$ from a vanishing theorem of N. Mok, which
contribute to a priori information about the sheaf cohomologies in
Corollary \ref{algebraic description}. We first recall the following
\begin{definition}\label{properly seminegative}
Let $M$ be a complex manifold and $(F,h)$ a hermitian holomorphic
vector bundle over $M$. Let $\Theta$ be the curvature tensor of
$(F,h)$. $(F,h)$ is said to be semi-positive (in the sense of
Griffiths) at the point $x\in M$, if for any non-zero tangent vector
$v\in T_{X,x}$ and any non-zero vector $e\in F_x$, $\Theta(e,\bar
e,v,\bar v)\geq 0$. It is said to be properly semi-positive if
furthermore for certain non-zero vectors $e_0$ and $v_0$ and has
$\Theta(e_0,\bar e_0, v_0, \bar v_0)=0$.
\end{definition}
The significance of the notion \emph{properly semi-positive} in the
case of locally hermitian symmetric domains lies in the following
theorem due to N. Mok, as a direct consequence of his metric
rigidity theorem on irreducible locally homogenous bundles:
\begin{theorem}[Mok]\label{vanishing result of Mok}
Let $\Omega$ be a hermitian symmetric domain and $\Gamma$ a
torsion-free discrete subgroup of $G=Aut^0(\Omega)$ such that the
quotient $M=\Gamma \backslash \Omega$ is irreducible (i.e., not a
product of two complex manifolds of positive dimensions) and has
finite volume with respect to the canonical metric. Let $(F,h)$ be a
non-trivial irreducible locally homogenous vector bundle over $M$
with an invariant hermitian metric $h$. If $(F,h)$ is properly
semi-positive at one point (and hence for all points), then
$H^0(X,F)=0$.
\end{theorem}
\begin{proof}
See the beginning of \S2.3 Chapter 10, \cite{Mo} which is a direct
application of Theorem 1 \S2.2 in loc. cit..
\end{proof}
The following proposition gives a sufficient condition for a
semi-positive locally homogenous bundle being properly
semi-positive.
\begin{proposition}\label{criterion of properly semipositive}
Let $M$ be as above, and $(F,h)$ an irreducible semi-positive
locally homogenous bundle over $M$. If there is a Higgs bundle
$(E,\theta)$ corresponding to a locally homogenous VHS $\W$ over $M$
such that $F\subset E^{p,q}$ and there exist non-zero vectors $e \in
F_{0}$ and $v \in T_{X,0}$ such that $\theta_{v}(e)=0$ (where
$\theta_{v}(e)$ is given by $\theta(e)(v)$), then $(F,h)$ is
properly semi-positive.
\end{proposition}
\begin{proof}
Because $F$ is an irreducible component of $E^{p,q}$, the second
fundamental form vanishes (see Chapter I in \cite{Ko}). So we can
use Griffiths curvature formula (see Lemma 7.18 in \cite{Sch} for
example) for $E^{p,q}$ to calculate the curvature of $F$:
$$
\Theta(e,\overline{e'})=(\theta(e), \theta(e'))-(\theta^{\dag}(e),
\theta^{\dag}(e')) \ \ \forall e,e' \in E^{p,q},
$$
where $\theta^{\dag}$ is the $\Hom(E^{p,q},E^{p+1,q-1})$-valued
$(0,1)$-form determined by requiring $\theta^{\dag}_{\bar v}$ be the
adjoint of $\theta_v$ relative to the Hodge metric. By assumption,
$\theta_{v}(e)=0$. So
$$
\Theta(e,\bar e, v,\bar{v})=(\theta_{v}(e),
\theta_{v}(e))-(\theta^{\dag}_{\bar{v}}(e),
\theta^{\dag}_{\bar{v}}(e))=-(\theta^{\dag}_{\bar{v}}(e),
\theta_{\bar{v}}^{\dag}(e))\leq 0.
$$
Since $F$ is semi-positive, $\Theta(e,\bar e, v,\bar{v})\geq 0$.
Therefore $\Theta(e,\bar e, v,\bar{v})=0$.
\end{proof}
Note that the line bundle $\sC_I$ is of the form
$\bigotimes_{i=1}^{n}\sL_{i}^{s_i}$ for an $n$-tuple of integers
$(s_1,\ldots,s_n)$.
\begin{proposition}\label{Mok vanishing}
For an $n$-tuple $(s_1,\ldots,s_n) \in \N_{0}^n \setminus
\{(0,\ldots,0)\}$ one has the vanishing
$$
H^0(X,\bigotimes_{i=1}^{n}\sL_{i}^{s_i}|_{X})=0
$$
if one of the entries $s_i$ is zero.
\end{proposition}
\begin{proof}
Take $m=(s_1,\ldots,s_n)$. By the discussion at the beginning of
\S\ref{logarithmic higgs cohomology}, it is clear that
$\bigotimes_{i=1}^{n}\sL_{i}^{s_i}|_X$ is the first Hodge bundle,
namely the $E^{|m|,0}$ part of the corresponding Higgs bundle to
$\V_m$. It follows from the Griffiths curvature formula as in the
proof of Proposition \ref{criterion of properly semipositive} that
$E^{|m|,0}$ is semi-positive. Assume $s_1=0$ without loss of
generality. Let $e_{i}$ be a nonvanishing local section of $\sL_i$
at $z_0$. Via the natural projection map, $X$ and $\HH^{n}$ are
analytically local isomorphic to each other. A generator of the
tangent subspace $T_{\HH_1,i}\subset T_{\HH^{n},z_0}$ gives rise to
a tangent vector $v$ of $T_{X,0}$. By Proposition \ref{basic
logarithmic higgs bundles}, the Higgs field $\theta_{i,1}$ along the
direction $v$ acts trivially on $e_{i}$ for $2\leq i\leq n$. For
$\theta_m=\bigotimes_{i=2}^{n}\theta_{i,1}^{ s_i}$, it follows that
for the local section $e=\bigotimes_{i=2}^{n}e_i^{ s_i}$ of
$\bigotimes_{i=1}^{n}\sL_{i}^{s_i}$, $\theta_{m,v}(e)=0$. By
Proposition \ref{criterion of properly semipositive} and Theorem
\ref{vanishing result of Mok}, the proposition follows.
\end{proof}
We remark that the condition on the irreducibility of $\Gamma$ in
the above result is crucial: for a product of modular curves the
analogous statement does not hold.

\section{$L^2$-cohomology of local systems over Hilbert modular variety and the pure Hodge structure}
\label{intersection cohomology} Let $g$ (resp. $h$) be the group
invariant metric on $X$ (resp. $\V_m$) (unique up to constant), and
let $H^k_{(2)}(X^*,\V_m)$ be the $L^2$-cohomology group of degree
$k$ with coefficients in $\V_m$ with respect to the invariant
metrics. Our principal aim in this section is to study the pure Hodge
structure on $H^k_{(2)}(X^*,\V_m)$ given by theory of $L^2$-harmonic forms.

Let us remind the reader that the local system
$\V_m$ in our case has been assumed to be non-trivial from the
beginning, i.e., $m\neq 0$. Let $D$ be the $C^{\infty}$-
Gau{\ss}-Manin connection and its $(1,0)$-part and $(0,1)$-part
decomposition reads
$$ D = d' + d''. $$ Let $\bar \partial$ be the holomorphic
structure of $E_m$ and $\partial$ the $(1,0)$-part of the metric
connection of the Hodge metric. We define
$$ \ D'=\partial +\bar \theta, \ D''=\bar \partial +
\theta. $$ Our notation follows \S4 in \cite{Sim} which differs
unfortunately from Zucker \cite{Zuc1}. On page 256 loc. cit. the
operators $D',D'', d'_{\rho}, d^{''}_{\rho}$ are defined, which are
$\partial, \bar \partial, \theta, \bar\theta$ respectively. The
operators $\mathfrak{D}', \mathfrak{D}''$ defined on page 265 loc.
cit. are our $D',D''$ respectively. For any of the above operators
$\mathcal{D}$, we denote its $L^{2}$-adjoint by $\mathcal{D}^{*}$
and Laplacian operator by $\square_{\mathcal{D}}$.

For a subset $I\subset \{1 ,\ldots,n\}$, we denote by $\sC_I^0$ the
restriction of $\sC_I$ to $X$, which is a locally homogenous line
bundle over $X$, and by $H^*_{(2)}(X,\sC_I^0)$ the $L^2$-Dolbeault
cohomology of $\sC_I^0$ with respect to the \emph{invariant} metric
$g$ of $X$ and the group invariant metric $h$ on $\sC_I^0$. That is,
it is the cohomology of the complex
$(A^{0,*}_{(2)}(X,\sC_{I}^0),\bar
\partial)$, where $A^{0,i}_{(2)}(X,\sC_I^0)$ is the space of $L^2$-integrable
smooth $(0,i)$-forms $\phi$ over $X$ with coefficients in $\sC_I^0$
such that $\bar \partial\phi$ is $L^2$. Let $\mathfrak h^i_{\bar
\partial,(2)}(X,\sC_I^0)\subset A^{0,i}_{(2)}(X,\sC_I^0)$ be the
subspace of $L^2$-$\bar
\partial$-harmonic forms. By the standard $L^2$-harmonic theory, one
has a natural injection $\mathfrak h^i_{\bar
\partial,(2)}(X,\sC_I^0)\to H^i_{(2)}(X,\sC_I^0)$, and it is an
isomorphism iff the range of $\bar \partial$ is closed (in the
Hilbert space of weakly differentiable $L^2$-forms). But we know
from the finite dimensionality of $H_{(2)}^*(X^*,\V_m)$ that the
range of $D$ is closed, and hence the range of $\square_{D}$ is
closed. As shown in page 609 \cite{Zuc3}
$\square_{D}=2\square_{D^{''}}$ holds in the strict operator sense,
it follows that $\square_{D^{''}}$ has closed range and thus the
range of $D^{''}$ is closed as well. Now that $\theta$ acts on
$\sC_I^{0}$ by zero, the range of $\bar
\partial$ is closed. In summary, one has a natural isomorphism
$$
H^i_{(2)}(X,\sC_I^0)\cong \mathfrak h^i_{\bar
\partial,(2)}(X,\sC_I^0),
$$
and particularly the finite dimensionality of
$H^i_{(2)}(X,\sC_I^0)$.
\begin{theorem}\label{vanishing of cohomologies of cohomology sheaf}
Let $I\subset \{1,\ldots,n\}$ be a subset. Then for $i\neq n-|I|$,
it holds that
$$
\dim H^i_{(2)}(X,\sC_I^0)=0.$$
\end{theorem}
In order to prove it, we need to employ the full machinery of
$L^2$-harmonic theory. Let $A^{n}_{(2)}(X,\V_m)$ (resp.
$A^{p,q}_{(2)}(X,\V_m)$) be the space of smooth $L^2$-$n$-forms
(resp. type $(p,q)$-forms) on $X$ with coefficients in $\V_m$.
Denote the subspace of $L^2$-harmonic forms by
$$
\mathfrak{h}_{(2)}^{n}(X,\V_m)= \{ \alpha \in A^{n}_{(2)}(X,\V_m) |
\square_{D} \alpha =0 \},
$$
similarly the subspace of $L^2$-harmonic $(p,q)$-forms by
$\mathfrak{h}_{(2)}^{p,q}(X,\V_m)$. By Zucker \cite{Zuc1}, one has
natural isomorphisms
$$
H_{(2)}^n(X^*,\V_m)\cong
\mathfrak{h}_{(2)}^{n}(X,\V_m)=\bigoplus_{p+q=n}\mathfrak{h}_{(2)}^{p,q}(X,\V_m).
$$
By abuse of notation, in the following we denote again by
$(E_m,\theta_m)$ the restriction of $(E_m,\theta_m)$ to $X$, which
is a Higgs bundle over $X$. We consider the cohomology sheaf of the
Higgs complex of $(E_m,\theta_m)$ over $X$ at the $i$-th place:
$$
E_m\otimes
\Omega^{i-1}_X\stackrel{\theta_m^{i-1}}{\longrightarrow}E_m\otimes
\Omega^{i}_X\stackrel{\theta_m^{i}}{\longrightarrow} E_m\otimes
\Omega^{i+1}_X,
$$
and denote by $\theta_m^{*,i}$ the adjoint of $\theta_m^i$ with
respect to the Hodge metric on $E_m$ and the invariant metric on
$X$. We have the following
\begin{lemma}\label{criterion of cohomology sheaf}
The $i$-th cohomology sheaf of the Higgs complex can be
characterized by those sections of $E_m\otimes \Omega^{i}_X$
satisfying the equations $\theta_m^i=\theta_m^{*,i-1}=0$.
\end{lemma}
\begin{proof}
Let $\sC=\frac{{\rm Ker}\ \theta_m^i}{{\rm Im}\ \theta_m^{i-1}}$ be
the cohomology sheaf. By the homogeneity, one has the holomorphic
and metric decomposition of Hermitian vector bundles
$$
{\rm Ker}\ \theta_m^i=\sC\oplus {\rm Im}\ \theta_m^{i-1}.
$$
So $\sC={\rm Ker}\  \theta_m^i\cap ({\rm Im}\
\theta_m^{i-1})^{\bot}$. On the other hand, one has clearly that
${\rm Ker}\ \theta_m^{*,i-1}= ({\rm Im}\ \theta_m^{i-1})^{\bot}$.
The lemma follows.
\end{proof}

\begin{lemma}\label{two spaces of harmonic forms}
For each $I$, one has the inclusion $\mathfrak h^i_{\bar
\partial,(2)}(X,\sC_I^0)\subset \mathfrak h^{|I|,i}_{(2)}(X,\V_m)$.
\end{lemma}
\begin{proof}
This follows from (5.22) and (5.14) in \cite{Zuc1}. We need to
explain the notation. By Proposition \ref{cohomology sheaves
matrix}, $\sC_I^0$ is a direct summand of the $|I|$-th cohomology
sheaf of the Higgs subcomplex $\Omega_{|m_I|+|I|}^*(E_m,\theta_m)$.
By (5.22) and (5.14) in \cite{Zuc1}, it follows that
$$
\mathfrak h^i_{\bar
\partial,(2)}(X,\sC_I^0)\subset \mathfrak h^{|I|,i}_{(2)}(X,\V_m).
$$
\end{proof}
Let $I=(i_1,\ldots,i_p),J=(j_1,\ldots,j_q)$ be two multi-indices
with $1\leq i_1<\cdots<i_p\leq n$ and $1\leq j_1<\cdots<j_q\leq n$.
Put $dz_{I}=dz_{i_1}\wedge \cdots\wedge dz_{i_p}$ and
$\overline{dz_{J}}=d\bar z_{j_1}\wedge\cdots\wedge d\bar z_{j_q}$.
Let $\mathfrak h_{(2)}(I,J;\V_m)$ be the subspace of $ \mathfrak
h^{p,q}_{(2)}(X,\V_m)$ consisting of those elements whose pull-back
to $\HH^{n}$ are of form $f_{I,\bar J}dz_I\wedge \overline{dz_J}$.
The following lemma is the $L^2$-analogue of Proposition 1.2 in
\cite{MS} and its proof holds verbatim for $L^2$-harmonic forms.

\begin{lemma}\label{fine decomposition of harmonic forms}
One has an orthogonal decomposition into a direct sum of subspaces
of $L^2$-harmonic forms for each $(p,q)$:
$$
\mathfrak h^{p,q}_{(2)}(X,\V_m)=\bigoplus_{\stackrel{I,J\subset
\{1,\ldots,n\},}{|I|=p,|J|=q}}\mathfrak h_{(2)}(I,J;\V_m).
$$
\end{lemma}

The next proposition gives the $L^2$-analogue of Proposition 4.1-4.3
in \cite{MS} in the case of non-trivial local systems.
\begin{proposition}\label{pq component of harmonic forms}
Let $I, J$ be as above. Then $\mathfrak h_{(2)}(I,J;\V_m)=0$ unless
$I\cup J$ (as a set) is equal to $\{1,\ldots,n\}$ and $I \cap J
=\emptyset$.
\end{proposition}

\begin{proof} Let $\alpha=\alpha_{I,\bar J}\in \mathfrak h_{(2)}(I,J;\V_m)$.
We prove the statement by induction on $q=|J|$. When $q=0$,
$\alpha\in A^{p,0}(X,\V_m)=A^{0,0}(X,E_m\otimes \Omega^p_X)$.
Because $\alpha$ is $D$-harmonic and equivalently $D''$-harmonic,
one has $\bar \partial (\alpha)=0$, $\theta(\alpha)=0$ and
$\theta^*(\alpha)=0$ (see Corollary 3.20 in \cite{Zuc1}). This
implies that $\alpha$ is a global holomorphic section of $E_m\otimes
\Omega^p_X$ and by Lemma \ref{criterion of cohomology sheaf} it is
even a global section of the cohomology sheaf of the Higgs complex
at the $p=|I|$-th place. Thus it must be a global section of
$\sC^0_I=\bigotimes_{i\in I}\sL_i^{m_i+2} \otimes \bigotimes_{i\in
I^c} \sL_i^{-m_i}|_{X}$ by consideration of $\{I,J;\V_m\}$-type.\\

{\bf Case 1:} $I\neq \{1,\ldots,n\}$ and $m_i=0$ for all $i\in I^c$.
In this case, $\sC_I^0$ is of the type in Proposition \ref{Mok
vanishing}. Since $\alpha\in \mathfrak h^0_{\bar
\partial,(2)}(X,\sC_I^0)\subset
H^0(X, \sC_I^0)$, $\alpha=0$ by Proposition \ref{Mok vanishing}.\\

{\bf Case 2:} $m_i\neq 0$ for certain $i\in I^c$. In this case, we
recall the decomposition of the differential operator $D$ over the
space $A^{p}_{(2)}(X,\V_m)$ 
$$
D=D'+D''=d'+d''; \ D'=\partial +\bar \theta, \ D''=\bar \partial +
\theta,\ d'=\partial +\theta, d''=\bar \partial + \bar \theta.
$$
$\Box_{D}(\alpha)=0$ implies that $d^{''}(\alpha)=0$. Since
$\bar\partial \alpha=0$, it follows that $\bar \theta(\alpha)=0$.
Now we take a smooth open neighborhood $U$ of $0\in X$ with the
local coordinates $\{z_1,\ldots,z_n\}$ and $e$ be a holomorphic
basis of the line bundle $C_{I}^0$ on $U$. Write
$\alpha=f(z_1,...,z_{n})e$ over $U$. Assuming the following claim,
one has
$$
\bar \theta (\alpha)=\bar \theta(f(z)e)=\bar f(z) \bar \theta(e)=0,
$$
and then $f=0$. So $\alpha=0$. The following claim is a direct
consequence of Proposition \ref{basic logarithmic higgs bundles}.
\begin{claim} One has $\bar{\theta}(e) \neq 0$, where $e$ is as above and
$$
\bar \theta:  A_{(2)}^{(p,0)}(E^{|m_I|,|m|-|m_I|})\to
A_{(2)}^{(p,1)}(E^{|m_I|+1,|m|-|m_I|-1}).
$$
\end{claim}
\begin{proof}
Let $\theta_i=\theta_{\partial_{z_i}}$ be the Higgs field along the
tangent direction $\partial_{z_i}$. Let $e_i$ (resp. $e_i^*$) be a
local basis of $\sL_i$ (resp. $\sL_i^{-1}$) over $U$ such that
$\theta_i(e_i)=e_i^*$. By Proposition \ref{basic logarithmic higgs
bundles}, one has $\theta_i(e_j)=0$ for $j\neq i$, and of course
$\theta_i(e_j^*)=0$ for any $j$. Put $e_I^{m_I}=\bigotimes_{i\in
I}e_i^{\otimes m_I}$ and $e_{I^c}^{*m_{I^c}}=\bigotimes_{i\in
I^c}e_i^{*m_i}$. Then one has $e=dz_I\otimes e_I^{m_I}\otimes
e_{I^c}^{*m_{I^c}}$ up to an invertible holomorphic function, which
does not affect the proof. Recall the local formula of $\bar \theta$
(see \S1 in \cite{Sim}):
$$
\bar \theta=\sum_{i}\bar \theta_i d\bar{z}_i,
$$
where $\bar \theta_i$ is the adjoint of the matrix $\theta_i$ with
respect to the Hodge metric. By the product rule for the Higgs field
with respect to tensor products (see \S1 in \cite{Sim}), one has for
$i\in I^c$,
\begin{eqnarray*}
  \bar \theta_i(dz_I\otimes e_I^{m_I}\otimes
e_{I^c}^{*m_{I^c}})&=& dz_I\otimes \bar \theta_i(e_I^{m_I})\otimes
e_{I^c}^{*m_{I^c}}+dz_I\otimes e_I^{m_I}\otimes \bar \theta_i
(e_{I^c}^{*m_{I^c}}) \\
   &=& dz_I\otimes e_I^{m_I}\otimes \bar \theta_i
(e_{I^c}^{*m_{I^c}})\\
    &=& \frac{m_i}{4y_i^2}\cdot dz_I\otimes e_I^{m_I}\otimes e_{I^c\backslash \{i\}}^{*m_{I^c\backslash \{i\}}}\otimes (e_i^{*m_i-1}\otimes
    e_i).
\end{eqnarray*}
Similarly one gets $ \bar \theta_i(dz_I\otimes e_I^{m_I}\otimes
e_{I^c}^{*m_{I^c}})=0$ for $i\in I$. Thus one has the formula
$$
\bar \theta(e)=\sum_{i\in I^c,m_i\neq
0}\frac{m_i}{4y^2_i}(dz_I\wedge d\bar{z}_i\otimes e_I^{m_I}\otimes
e_{I^c\backslash \{i\}}^{*m_{I^c\backslash \{i\}}}\otimes
e_i^{*m_i-1}\otimes e_i).$$ By the assumption of the case, the above
expression is non-zero. The claim is proved.
\end{proof}

In summary, the space $\mathfrak h_{(2)}(I,\emptyset;\V_m)$ is zero
unless $I$ is the whole set. This proves the $q=0$ case. Now we
assume $q>0$. There are two possibilities, namely the case $I\cap
J\neq \emptyset$ and the case $I\cap J=\emptyset$. Consider the
former case. Let $\Lambda$ be the adjoint of the Lefschetz operator
$L=\wedge \omega$ on the space of differential forms, where $\omega$
is the K\"{a}hler form of the metric $g$ on $X$. By the standard
$L^2$-harmonic theory, $\Lambda(\alpha)$ is again an $L^2$-harmonic
form. As $I\cap J\neq \emptyset$, $\Lambda (\alpha)=0$ if and only
if $\alpha=0$. Write $\Lambda(\alpha)=\sum_{I',J'}\beta_{I',\bar
J'}$. Then $|J'|=q-1$ and $I' \cup J'$ is not equal to
$\{1,...,n\}$. So one proves by induction that each term
$\beta_{I',\bar J'}$ of $\Lambda(\alpha)$ is zero and hence
$\alpha=0$. Consider the latter case. Let $\HH^{n}_{\bar J}$ be the
complex manifold whose underlying riemannian structure is the same
as that of $\HH^{n}$ but the complex structure differs from the
usual one by that at the $j$-th factor for $j\in J$, one takes the
complex conjugate complex structure of $\HH$ (see \S4 in \cite{MS}).
One puts $X_{\bar J}=\Gamma\backslash \HH^{n}_{\bar J}$. As observed
in \S4 \cite{MS}, such an operation identifies the space of
$L^2$-harmonic forms $\mathfrak{h}^n_{(2)}(X,\V_m)$ with
$\mathfrak{h}^n_{(2)}(X_{\bar J},\V_m)$, but maps the subspace of
type $\{I,J;\V_m\}$ on $X$ to the subspace of type $\{I\cup
J,\emptyset;\V_m\}$ on $X_{\bar J}$. This allows us to reduce the
proof to the case $q=0$ on $X_{\bar J}$. However the above
arguments, particularly the truth of Proposition \ref{Mok
vanishing}, holds for $X_{\bar J}$ as well. Therefore the second
case also follows.
\end{proof}

Now we can proceed to the proof of Theorem \ref{vanishing of
cohomologies of cohomology sheaf}.
\begin{proof}
Let $\alpha\in \mathfrak h^i_{(2),\bar
\partial}(X,\sC_I)$ be the $L^2$-harmonic representative of a non-zero
cohomology class of $H^i_{(2)}(X,\sC_I^0)$. By Lemma \ref{two spaces
of harmonic forms}, $\alpha\in \mathfrak h^{p,q}_{(2)}(X,\V_m)$
where we rewrite $p=|I|,q=i$. By Lemma \ref{fine decomposition of
harmonic forms}, one writes
$\alpha=\sum_{I,J,|I|=p,|J|=q}\alpha_{I,\bar J}$ into sum of
$L^2$-harmonic forms with $\alpha_{I,\bar J}$ type $\{I,J;\V_m\}$.
By Proposition \ref{pq component of harmonic forms}, $\alpha_{I,\bar
J}=0$ unless $J$ is the complement of $I$. This implies that if
$q\neq n-p$, namely $i\neq n-|I|$, $\alpha=0$ and hence the theorem.
\end{proof}
In the course of the above proof, we have introduced for a subset
$J\subset\{1,\cdots,n\}$ a complex manifold $X_{\bar J}$. It is
again a discrete quotient of $\HH^n$: Putting $\Gamma_{J}$ to be
$g_{J}\Gamma g_{J}^{-1}$ where $g_J=(g_1,\cdots,g_n)\in
GL_2(\R)^{\times n}$ with
$$
g_i=\left\{
  \begin{array}{ll}
    \left(
                           \begin{array}{cc}
                             1 & 0 \\
                             0 & -1 \\
                           \end{array}
                         \right), & i\in J; \\
     id, & i\notin J.
  \end{array}
\right.,
$$
one sees easily that $X_{\bar J}=\Gamma_J\backslash \HH^n$. By
Margulis's arithmeticity theorem (see \cite{Ma}), $\Gamma_J\subset
G(\R)$ is again an arithmetic lattice. One defines the locally
homogenous line bundle $\sL_{i}$ over $X_{\bar J}$ as well as its
good extension to a smooth toroidal compactification $\bar{X}_{\bar
J}$ of $X_{\bar J}$. As a byproduct of the previous proof, we obtain
the following result:
\begin{proposition}\label{byproduct}
$$
\dim H^{n-|I|}_{(2)}(X,\sC_{I}^0)=\dim H^0_{(2)}(X_{\bar I^c},
\bigotimes_{j=1}^{n}\sL_{j}^{m_j+2}|_{X_{\bar I^c}}).
$$
\end{proposition}
\begin{proof}
Using the trick of taking the conjugate complex structures on
$\HH^{n}$ at the factors $I^c$ (see the proof of Proposition \ref{pq
component of harmonic forms}), one obtains an identification of the
space of harmonic forms of type $\{I,I^c;\V_m\}$ on $X$ with that of
type $\{(1,\ldots,n),\emptyset; \V_m\}$ on $X_{\bar I^c}$.
\end{proof}
The space of $L^2$-sections admits a natural analytical description
using a smooth toroidal compactification.
\begin{proposition}\label{L2 condition is algebraic}
For each $J\subset \{1,\cdots,n\}$, one has a natural isomorphism
$$
H_{(2)}^0(X_{\bar J}, \bigotimes_{j=1}^{n}\sL_{j}^{m_j+2}|_{X_{\bar
J}})\cong H^0(\bar{X}_{\bar J}, \sO_{\bar{X}_{\bar J}}(-S_{\bar
J})\otimes \bigotimes_{j=1}^{n}\sL_{j}^{m_j+2}),
$$
where $S_{\bar J}=\bar{X}_{\bar J}-X_{\bar J}$ is the boundary
divisor.
\end{proposition}
\begin{proof}
Obviously it suffices to show the statement for $J=\emptyset$, i.e.,
for $X_{\bar J}=X$. Let $\iota: X\to \bar X$ be the inclusion. One
has the relation
$$
H_{(2)}^0(X, \sC^0_{\{1,\ldots,n\}})\subset
A^{0,0}_{(2)}(X,\sC^0_{\{1,\ldots,n\}}) \subset A^{0}(\bar X,
\iota_{*}\sC^0_{\{1,\ldots,n\}}).
$$
We define $\Omega^{0}_{(2)}(\sC^0_{\{1,\ldots,n\}})$ to be the
subsheaf of $\iota_{*}\sC^0_{\{1,\ldots,n\}}$ consisting of germs of
$L^{2}$-holomorphic sections. It is clear that one has a natural
isomorphism
$$
H_{(2)}^0(X, \sC^0_{\{1,\ldots,n\}})\cong H^0(\bar
X,\Omega^{0}_{(2)}(\sC^0_{\{1,\ldots,n\}})).
$$
Recall that
$\sC^0_{\{1,\ldots,n\}}=\bigotimes_{j=1}^{n}\sL_{j}^{m_j+2}|_X$. In
the following we show that
$$
\Omega^{0}_{(2)}(\bigotimes_{j=1}^{n}\sL_{j}^{m_j+2}|_X)=\sO_{\bar
X}(-S)\otimes \bigotimes_{j=1}^{n}\sL_{j}^{m_j+2}.
$$
The question is local at infinity, and we shall only show the case
where $m=(1,0,\ldots,0)$, since the proof in general case is
completely analogous. It suffices also to consider a small
neighborhood $U$ of a maximal singular point $P\in S$. Let
$\{u_1,\ldots,u_n\}$ be a set of local coordinates of $U$ such that
$S\cap U$ is defined by $\prod_{i=1}^{n}u_i=0$. We also put
$\omega_{pc}=\sum_{i=1}^{n}\frac{\sqrt{-1}}{\pi}\frac{du_{i} \wedge
d\bar{u}_i}{|u_{i}|^{2}|\log|u_{i}||^{2}}$, the Poincar\'{e} metric
on $U-S\cong (\triangle^*)^{n}$, and
$\omega=c_1\sum_{i=1}^{n}\frac{dz_{i} \wedge d\bar{z}_i}{y^{2}_{i}}$
(in the following the $c_i$ are certain constants), the invariant
metric. Their volume forms are computed respectively by
$$
{\rm Vol}_{\omega_{pc}} = \frac{c_2\bigwedge_{i=1}^{n}(du_{i}\wedge
d\bar{u}_i)}{\prod_{i=1}^{n}|u_{i}|^{2}|\log|u_{i}||^{2}},\quad {\rm
Vol}_{\omega}= \frac{c_3\bigwedge_{i=1}^{n}(dz_{i}\wedge
d\bar{z}_i)}{\prod_{i=1}^{n}y_i^{2}}.
$$
By the theory of toroidal resolutions of a cusp singularity (see
\cite{AMRT} and \cite{El}), one has the following formula for the
change of local coordinates:
$$
2\pi\sqrt{-1}\cdot z_{i}= \sum_{j=1}^{n} a_{i,j} \log u_{i},
$$
where $a_{i,j}>0$ for all $i,j$. Comparing the real parts of the
above equality, one obtains
$$
y_{i}= \sum_{j=1}^{n} -a'_{i,j}\log |u_{i}|
$$ with $a'_{i,j} =\frac{a_{i,j}}{2\pi}$. The estimate of Hodge norms in
Theorem 5.21, \cite{CKS} is taken over the following type of region
$$
D_{\epsilon}=\{(u_{1},\ldots, u_{n}) \in (\triangle^{*})^{n} |\
\frac{\log|u_{1}|}{\log|u_{2}|} > \epsilon,\ldots,
\frac{\log|u_{n-1}|}{\log|u_{n}|}>\epsilon, -\log|u_{n}|>\epsilon \}
$$
for some $\epsilon>0$. For an element $\sigma\in S_n$, the
permutation group over $n$ elements, we put $D_{\sigma}$ to be the
region obtained by permutating the indices of $\{u_i\}$ in the
definition of $D_{\epsilon}$. So $D_{id}=D_{\epsilon}$, and note
that $\{D_{\sigma}\}_{\sigma\in S_n}$ is an open covering of a small
neighborhood of $P$ for suitable chosen $\epsilon$. By shrinking $U$
if necessary, it covers $U$. It is clear that the
square-integrability over $U$ is equivalent to that over each
$D_{\sigma}$. Now let $v_i$ be a local trivializing section of
$\sL_i$ over $U$, and write $f(u_1,\ldots,u_n)v_1^{3}\otimes
v_{2}^{2}\otimes \cdots \otimes v_{n}^{2}$ for an element in
$\iota_{*}(\sL_{1}^{3}\otimes \sL^{2}_{2}\otimes \cdots \otimes
\sL^{2}_{n})|_X(U)$. By Theorem 5.21, \cite{CKS}, $||v_{i}||^{2}
\sim |\log|u_{1}||$ over $D_{id}$. The condition that
$fv_1^{3}\otimes v_{2}^{2}\otimes \cdots \otimes v_{n}^{2}$ being
$L^2$ over $D_{id}$ means that
$$
\int_{D_{id}}|f|\cdot||v_1^{3}\otimes v_{2}^{2}\otimes \cdots
\otimes v_{n}^{2}||{\rm Vol}_{\omega}<\infty.
$$
Since over $D_{id}$,
$$
\frac{{\rm Vol}_{\omega}}{{\rm Vol}_{\omega_{pc}}} \sim
\frac{\prod_{i=1}^{n}|\log |u_{i}||^{2}}{\prod_{i=1}^{n}y_{i}^{2}}
=\frac{\prod_{i=1}^{n}|\log
|u_{i}||^{2}}{\prod_{i=1}^{n}(\sum_{j=1}^{n} -a'_{i,j}\log
|u_{i}|)^{2}} \sim  \frac{\prod_{i=1}^{n}|\log |u_{i}||^{2}}{|\log
|u_{1}||^{2n}}
$$
holds, it follows that
\begin{eqnarray*}
   \int_{D_{id}}|f|\cdot||v_1^{3}\otimes v_{2}^{2}\otimes \cdots
\otimes v_{n}^{2}||{\rm Vol}_{\omega} &\sim&  \int_{D_{id}}
|f|\cdot|\log|u_1||^{2n+1}
{\rm Vol}_{\omega}\\
    &\sim&\int_{D_{id}}
|f|\cdot|\log|u_1||^{3}\cdot\prod_{i=1}^{n}|\log|u_{i}||^{2}{\rm
Vol}_{\omega_{pc}} \\
&\sim& \int_{D_{id}}
\frac{|f|\cdot|\log|u_1||^{3}}{\prod_{i=1}^{n}|u_{i}|^{2}}\bigwedge_{i=1}^{n}(du_i\wedge
d\bar u_i).
\end{eqnarray*}
It is clear now that over $D_{id}$, the above section is $L^2$ if
and only if $f=u_1\cdot f'$ for certain holomorphic $f'$. Running
the above arguments for the other regions $D_{\sigma}$, one knows
that the section is $L^2$ over all $D_{\sigma}$ if and only if
$f=(u_1,\ldots,u_n)\cdot f''$ for certain holomorphic $f''$. This
shows the equality
$$
\Omega^{0}_{(2)}(\bigotimes_{j=1}^{n}\sL_{j}^{m_j+2}|_X)=\sO_{\bar
X}(-S)\otimes \bigotimes_{j=1}^{n}\sL_{j}^{m_j+2}
$$
over $U$. By the previous discussion, the above equality actually
holds over $\bar X$, which shows the proposition.
\end{proof}

\begin{corollary}\label{vanishing of intersection cohomology}
For $k\neq n$, $H_{(2)}^k(X^*, \V_m)=0$.
\end{corollary}
\begin{proof}
By Zucker (5.14),(5.22) in \cite{Zuc1} and
the finite dimensionality
of $H_{(2)}^{*}(X, \sC^0_{I})$, one has the equality
$$
\dim IH^k(X^*, \V_m)=
\dim
H_{(2)}^k(X^*,\V_m)=\sum_{P=0}^{k+|m|}\sum^{n}_{l=0}\sum_{\stackrel{I\subset
\{1,\ldots,n\},}{|I|=l,\ |m_I|+|I|=P}}\dim H_{(2)}^{k-l}(X,
\sC^0_{I}).
$$
In the above formula, it is clear that for $0\leq k\leq n-1$ fixed
and for all $l$, $k-l<n-l=n-|I|$ holds. By Theorem \ref{vanishing of
cohomologies of cohomology sheaf}, it follows that for $0\leq k\leq
n-1$, each direct summand in the right hand side of the formula is
zero, and hence $IH^k(X^*, \V_m)=0$. By Poincar\'{e} duality for
intersection cohomology, $IH^k(X^*, \V_m)$ vanishes for $n+1\leq
k\leq 2n$ as well. Hence the result holds for the $L^2$-cohomology groups.
\end{proof}

\begin{corollary}\label{hodge decomp of inter cohomlogy}
The Hodge decomposition of $H_{(2)}^n(X^*, \V_m)$ reads
$$
H_{(2)}^n(X^*, \V_m)=\bigoplus_{P+Q=n+|m|} H_{(2)}^{P,Q},
$$
where
$$
H_{(2)}^{P,Q}\cong \bigoplus^{n}_{l=0}\bigoplus_{\stackrel{I\subset
\{1,\ldots,n\},}{|I|=l,\ |m_I|+|I|=P}}H_{(2)}^{n-l}(X, \sC^0_{I}).
$$
In particular,
$$H_{(2)}^{n+|m|,0}\cong H^0(\bar X, \sO_{\bar
X}(-S)\otimes \bigotimes_{j=1}^{n}\sL_{j}^{m_j+2}),\quad
H_{(2)}^{0,n+|m|}\cong H^n(\bar X,\bigotimes_{i=1}^{n}\sL_i^{-m_i}).
$$
\end{corollary}
\begin{proof}
Continuing the arguments in the above proof, one sees that for each
$P$,
$$
H_{(2)}^{P,Q}\cong \bigoplus^{n}_{l=0}\bigoplus_{\stackrel{I\subset
\{1,\ldots,n\},}{ |I|=l,\ |m_I|+|I|=P}}H^{n-l}_{(2)}(X, \sC^0_{I}).
$$
It is clear that, for $P=0$ (resp. $P=n+|m|$), the above expression
consists of the unique term $H^n_{(2)}(X, \sC^0_{\emptyset})$ (resp.
$H_{(2)}^0(X, \sC^0_{\{1,\ldots,n\}})$). By Proposition \ref{L2
condition is algebraic}, $H_{(2)}^{n+|m|,0}\cong H^0(\bar X,
\sO_{\bar X}(-S)\otimes \bigotimes_{j=1}^{n}\sL_{j}^{m_j+2})$. By
Serre duality, one has a natural isomorphism $H_{(2)}^{0,n+|m|}\cong
H^n(\bar X,\bigotimes_{i=1}^{n}\sL_i^{-m_i})$.
\end{proof}

\section{Eisenstein cohomology of a Hilbert modular group}\label{Eisenstein cohomology}

Let $X^\sharp$ be the Borel-Serre compactification of $X$ with the
boundary $\partial X^\sharp$. Recall that $X$ is homotopy equivalent
to $X^\sharp$ and hence one has the natural restriction map $r:
H^*(X,\V_{\R})\to H^*(\partial X^\sharp,\V_{\R})$. The theory of
Eisenstein series (see \cite{Ha2} and \cite{Schw}) provides the
following space decomposition
$$
H^*(X,\V_{\R})=H_{!}^*(X,\V_{\R})\oplus H_{\Eis}^*(X,\V_{\R}),
$$
where $H_{!}^*(X,\V_{\R})$ is the image of the cohomology with
compact support, and $H_{\Eis}^*(X,\V_{\R})$ maps isomorphically to
the image of $r$. Its elements can be represented using Eisenstein
series. In this section, we study the Eisenstein cohomology
$H_{\Eis}^*(X,\V_{\R})$. Before doing anything, we first recall the
following result, which is a special case of the main theorems in
\cite{LS}:
\begin{theorem}[Li-Schwermer \cite{LS}] \label{Li-Schwermer}
If $\V_m$ is regular, then $H^i(X,\V_m)=0$ for $0\leq i\leq n-1$,
and $H^i(X,\V_m)=H_{\Eis}^i(X,\V_{m})\stackrel{r}{\cong}H^i(\partial
X^\sharp,\V_{m})$ for $i\geq n+1$.
\end{theorem}
The following lemma is known by (6.13-18) in \cite{Zuc2}:
\begin{lemma}\label{lemma of Zucker}
$H^i(\partial X^\sharp,\V_{m})=0$ unless $m_1=\cdots=m_n$. As a
consequence, $H_{\Eis}^i(X,\V_{m})=0$ if the relation
$m_1=\cdots=m_n$ is not satisfied.
\end{lemma}
The main result of this section is the following
\begin{theorem}\label{type of Eisenstein coho}
Assume $m_1=\cdots=m_n$ and $l\geq n$. Then the restriction map $r:
H^l(X,\V_{\R})\to H^l(\partial X^\sharp,\V_{\R})$ is surjective and
$\dim H_{\Eis}^l(X,\V_{m})=\binom{n-1}{l-n}h$. Moreover,
$$H_{\Eis}^l(X,\V_{m})\subset F^{|m|+n}H^l(X,\V_{m})$$ holds, where
$F^{.}$ is the Hodge filtration on $H^l(X,\V_{m})$.
\end{theorem}
After the paper was posted, Wildeshaus informed us that the main
result in \cite {BW} also shows the statement about the Hodge type
of the Eisenstein cohomology in the above theorem. The current
argument in the proof is based on the treatment of the Eisenstein
cohomology for constant coefficients in \cite{Fr} (see \S3 and \S4,
Ch. III in loc. cit.). In the following we assume $m_1=\cdots=m_n$.
It is clear that the proof of the theorem can be reduced to the
statement for the standard cusp $\infty$, which is the
$\Gamma$-equivalence class of $(\infty,\ldots,\infty)$. From now on
we pretend that $\partial X^\sharp$ has only one component. Let
$\Gamma_{\infty}\subset \Gamma$ be the stabilizer of $\infty$ and
put $X_{\infty}=\Gamma_{\infty}\backslash \HH^{n}$.
\begin{proof} We divide the whole proof into several steps.

{\bf Step 1}: A basis of $H^l(X_\infty,\V_m)$ for $n\leq l\leq
2n-1$. Let $X_{\infty}(1)$ be the quotient of the following set
$$
\{(z_1,\ldots,z_n)\in \HH^{n}|\ \prod_{i=1}^{n}y_i=1\}
$$
by $\Gamma_{\infty}$, which is naturally identified with $\partial
X^{\sharp}$. The group $\Gamma_{\infty}$ is of the form
$$
1\to U\to \Gamma_{\infty}\to M\to 1,
$$
where
\begin{eqnarray*}
  U &=& \Biggl\{ \left(
             \begin{array}{cc}
               1 & u^{(1)} \\
               0 & 1 \\
             \end{array}
           \right)\times \cdots\times \left(
                                        \begin{array}{cc}
                                          1 & u^{(n)} \\
                                          0 & 1 \\
                                        \end{array}
                                      \right)\in \Gamma_{\infty}|\
                                      u\in \sO_F
\Biggr\} \\
  M &=& \Biggl\{ \left(
             \begin{array}{cc}
               t^{(1)} & 0 \\
               0 & (t^{-1})^{(1)} \\
             \end{array}
           \right)\times \cdots\times \left(
                                        \begin{array}{cc}
                                           t^{(1)} & 0 \\
                                          0 & (t^{-1})^{(n)} \\
                                        \end{array}
                                      \right)\in \Gamma_{\infty}|\
                                      t\in \sO_F^*
\Biggr\}
\end{eqnarray*}
are free abelian groups of rank $n$ and $n-1$ respectively. The
$X_{\infty}(1)$ has two distinguished submanifolds: One is the
quotient of the set $\{(iy_1,\ldots,iy_n)\in \HH^{n}|\
\prod_{i=1}^{n}y_i=1\}$ by $M$, which is isomorphic to $(S^1)^{n-1}$
with 'coordinates' $\{\log y_1,\ldots,\log y_{n-1}\}$ and denoted
temporarily by $Y$), and the quotient of $\{(x_1+i,\ldots,x_n+i)\in
\HH^{n}|\ x_1,\ldots,x_n\in \R\}$ by $U$, which is isomorphic to
$(S^1)^{n}$ with 'coordinates' $\{x_1,\ldots,x_n\}$ and denoted
temporarily by $Z$. In fact, $X_{\infty}(1)$ is naturally a fiber
bundle with $Y$ (resp. $Z$) a section (resp. fiber) of it (see \S2
in \cite{Ha2}).
\begin{claim}\label{basis of boundary}
For $n\leq l\leq 2n-1$, the following set of vector valued
differential forms over $\HH^{n}$ is $\Gamma_{\infty}$-invariant and
defines a basis of $H^l(X_{\infty}(1),\V_m)$:
$$
\{\omega_a'=\frac{dy_{a}}{y_{a}}\wedge dx_1 \wedge \cdots \wedge
dx_n \otimes \bigotimes_{i=1}^{n}(e_{i1}+x_{i} e_{i2})^{m_1}|\
a\subset \{1,\ldots,n-1\},\ |a|=l-n\},
$$
where for $a=(i_1,\ldots,i_{l-n}),\ i_1<\cdots<i_{l-n}$,
$\frac{dy_{a}}{y_{a}} = \frac{dy_{i_{1}}}{y_{i_{1}}} \wedge \cdots
\wedge \frac{dy_{i_{i-n}}}{y_{i_{l-n}}}$, and $\Biggl\{
e_{i1}=\left(\begin{array}{c}
  1 \\
 0 \\
  \end{array}
  \right),\
e_{i2}=\left(
                                                        \begin{array}{c}
                                                         0 \\
                                                         1 \\
                                                        \end{array}
                                                      \right)\Biggr\}$
is the standard basis of $V_{i}$ at $z_0\in \HH^{n}$ (see
\S\ref{preliminaries}).
\end{claim}
\begin{proof}
Let $\pi: X_{\infty}(1)\to Y$ be the fibration. By \S2 in
\cite{Ha2}, the Leray spectral sequence of $\pi$ for
$H^l(X_{\infty}(1),\V_m)$ degenerates at $E_2$. By the theorem of
van Est (see \S2 in \cite{Ha2}), each grading
$H^*(Y,R^{l-*}\pi_*\V_m)$ is computed by its corresponding Lie
algebra cohomology. By the computations on the Lie algebra
cohomology in \S6 of \cite{Zuc2} (see particularly (6.18) and Lemma
(6.13)), one knows that
$$
H^l(X_{\infty}(1),\V_m)=H^{l-n}(Y,\C)\otimes H^{n}(Z,\V_m).
$$
Now it is straightforward to check that $\{\frac{dy_a}{y_a}|\
a\subset \{1,\ldots,n-1\},\ |a|=l-n\}$ provides a basis for
$H^{l-n}(Y,\C)$ and the element $\bigwedge_{i=1}^{n}dx_i \otimes
\bigotimes_{i=1}^{n}(e_{i1}+x_{i} e_{i2})^{m_1}$ is a basis for the
one dimensional space $H^{n}(Z,\V_m)$.
\end{proof}
Note that the inclusion $X_{\infty}(1)\subset X_{\infty}$ is a
homotopy equivalence. We claim the following
\begin{claim}\label{basis of X infty}
The following set of $\Gamma_{\infty}$-invariant vector valued
differential forms over $\HH^{n}$
$$
\{\omega_a = \frac{dz_{a}\wedge d\bar{z}_a}{y_{a}}\wedge dz_{b}
\otimes \bigotimes_{i=1}^{n}(e_{i1}+z_{i}e_{i2})^{m_1}|\ a\subset
\{1,\ldots,n-1\},\ |a|=l-n,\ b=a^c\}$$ defines a basis of
$H^{l}(X_{\infty}, \V_m)$.
\end{claim}
\begin{proof}
By Remark 3.1, Ch.III in \cite{Fr}, $\omega_a$ is cohomologous to
$$\frac{dy_{a}}{y_{a}}\wedge dx_1 \wedge \ldots \wedge dx_n \otimes
\bigotimes_{i=1}^{n}(e_{i1}+z_{i}e_{i2})^{m_1} $$ up to scalar. By
Claim \ref{basis of boundary} and the homotopy equivalence, it
remains to show that $$\frac{dy_{a}}{y_{a}}\wedge dx_1 \wedge \cdots
\wedge dx_n \otimes \bigotimes_{i=1}^{n}(e_{i1}+z_{i}e_{i2})^{m_1}$$
is cohomologous to $\omega_a'$ up to a scalar. Note that the
difference of the above two forms is a linear combination of forms
of the following type:
$$
\frac{dy_{a}}{y_{a}}\wedge dx_1 \wedge \cdots \wedge dx_n  \otimes
\bigotimes_{i=1}^{n}(y_{i}e_{i2})^{t_i}\otimes (e_{i1}+x_{i}
e_{i2})^{m_1-t_{i}},
$$
which is exact once one of the $t_{i}$ is positive. In fact, assume
$t_1\geq 1$ for example, the exterior differential of the following
form
$$
y_{1} \cdot \frac{dy_{a}}{y_{a}}\wedge \bigwedge_{i=2}^{n} dx_i
\otimes (y_{1} e_{12})^{t_{1}-1} \otimes (e_{11}+x_{1}
e_{12})^{m_1-t_{1}+1} \otimes
\bigotimes_{i=2}^{n}(y_{i}e_{i2})^{t_i}\otimes (e_{i1}+x_{i}
e_{i2})^{m_1-t_{i}}
$$
is up to a scalar equal to
$$
\frac{dy_{a}}{y_{a}}\wedge dx_1 \wedge \cdots \wedge dx_n  \otimes
\bigotimes_{i=1}^{n}(y_{i}e_{i2})^{t_i}\otimes (e_{i1}+x_{i}
e_{i2})^{m_1-t_{i}}.
$$
This shows the claim.
\end{proof}
{\bf Step 2}: Convergence of Eisenstein series. For each $\omega_a$,
we consider the following formal differential form $E(\omega_a)$ on
$X$ obtained by symmetrization (see \S3, Ch.III in \cite{Fr}):
$$
E(\omega_a)= \sum_{M \in \Gamma / \Gamma_{\infty}}  \omega_a | M,
$$
where $M=\left(
           \begin{array}{cc}
             a^{(1)} & b^{(1)} \\
             c^{(1)}  & d^{(1)}  \\
           \end{array}
         \right)\times \cdots\times \left(
           \begin{array}{cc}
             a^{(n)} & b^{(n)} \\
             c^{(n)}  & d^{(n)}  \\
           \end{array}
         \right)
$ runs through a set of representatives of $\Gamma /
\Gamma_{\infty}$, and $\omega_a |M= M^{*}\omega_a$ by considering
$M$ as transformation on $X_{\infty}$. If the above series
converges, then $E(\omega_a)$ defines a genuine vector valued
differential form on $X$. The following simple transformation
formulas
\begin{eqnarray*}
  dz_{i} |M &=& (c^{(i)}z_{i}+d^{(i)})^{-2}dz_i,  \quad d\bar{z}_{i} |M =(c^{(i)}\bar{z}_{i}+d^{(i)})^{-2}d\bar{z}_{i}, \\
  y_{i}|M &=& |c^{(i)}z_{i}+d^{(i)}|^{-2}y_{i},\quad (e_{i1} + z_{i}e_{i2}) | M = (c^{(i)}z_{i}+d^{(i)})^{-1} (e_{i1} +
z_{i}e_{i2})
\end{eqnarray*}
show that the series $E(\omega_a)$ obeys the relation $E(\omega_a) =
E_{\alpha,\beta}(z) \cdot \omega_a$, where
$$
E_{\alpha,\beta}(z)=  \sum_{M \in \Gamma / \Gamma_{\infty}}
\prod_{i=1}^n (c^{(i)}z_{i}+d^{(i)})^{-\alpha_{i}}
(c^{(i)}\bar{z}_{i}+d^{(i)})^{-\beta_{i}}
$$
with
\begin{equation*}
  \alpha_{i}   =  \begin{cases}
m_1+1 \indent {\rm if} \indent i\in a\\
m_1+2 \indent {\rm if} \indent i \in a^c
\end{cases},\quad \beta_{i}   =  \begin{cases}
1 \indent {\rm if} \indent i\in a\\
0 \indent {\rm if} \indent i \in a^c
\end{cases}.
\end{equation*}
This is the type of Eisenstein series considered in \cite{Fr}. Note
that for the constant coefficient they consider the border case
$r=1$, which requires the technique of Hecke summation to show the
convergence of the series. In the current case, Lemma 5.7, Ch.I in
\cite{Fr} shows that $E_{\alpha,\beta}(z)$ is absolutely convergent.
We show next that $E(\omega_a)$ is closed. For that, one considers
the Fourier expansion of $E_{\alpha,\beta}(z)$ at $\infty$, and as
argued in Proposition 3.3, Ch.III in \cite{Fr}, the key point is to
study the constant Fourier coefficient. The formula $a_0(y,s)$ in
Page 170, \cite{Fr} for $s=0,\ r=\frac{m_1+2}{2}$ and $\alpha_i,\
\beta_i$ as above shows that the constant Fourier coefficient of
$E_{\alpha,\beta}(z)$ is of the form $A +
\frac{B}{\prod_{i=1}^{n}y_{i}^{2r-1}}$ with $A=1,\ B=0$ (see Theorem
4.9, Ch.III in \cite{Fr}). One concludes from the proof of
Proposition 3.3, Ch.III in \cite{Fr} that $E(\omega_a)$ is closed
(hence a cohomology class in $H^{l}(X, \V_m)$), and it induces the
same cohomology class as $\omega_a$ in $H^{l}(X_{\infty}, \V_m)$.
The same proposition also shows that the restriction of
$E(\omega_a)$ to other cusps is zero. By the theory of Eisenstein
cohomology, $\{E(\omega_a)\}_{a\subset \{1,\ldots,n\},\ |a|=l-n}$
forms a basis of $H^l_{\Eis}(X,\V_m)$.

{\bf Step 3}: Hodge type of Eisenstein cohomology classes. By the
expression of $\omega_a$ in Claim \ref{basis of X infty}, one knows
that $\omega_a$ extends naturally to an element in
$A_{\bar{X}}^{n,0}(\log S)\wedge
\overline{A_{\bar{X}}^{|a|,0}(*S)}\otimes E^{|m|,0}_{m}$. Now by the
expression of the Eisenstein series $E_{\alpha,\beta}(z)$ in Theorem
4.9, \cite{Fr} for $A=1,\ B=0$, one sees that $E(\omega_a)$ lies
again in $A_{\bar{X}}^{n,0}(\log S)\wedge
\overline{A_{\bar{X}}^{|a|,0}(*S)}\otimes E^{|m|,0}_{m}$. The
expression of $E(\omega_a)$ shows that it is of logarithmic
singularity at infinity $S$. Therefore $E(\omega_a)$ represents a
cohomology class in $F^{|m|+n}H^{l}(X, \V_m)$.
\end{proof}

\section{The mixed Hodge structure on the cohomology groups}\label{final results}
Let $j: X\to X^*$ be the natural inclusion, and $H_{(2)}^{n}(X^*,\V_m)$ the $L^2$-cohomology group. Recall that by
the $L^2$-harmonic theory, one has the following natural isomorphism:
$$
H_{(2)}^n(X^*, \V_m)\cong
\mathfrak{h}_{(2)}^{n}(X,\V_m)=\bigoplus_{p+q=n}\mathfrak{h}_{(2)}^{p,q}(X,\V_m).
$$
Therefore, one has a natural map $H_{(2)}^{n}(X^*,\V_m) \rightarrow H^{n}(X,\V_m)$ by taking a cohomology class
to its $L^2$-harmonic representative. We assert the following

\begin{proposition} \label{injectivity}
The natural map $H_{(2)}^{n}(X^*,\V_m) \rightarrow H^{n}(X,\V_m)$ is
injective.
\end{proposition}
\begin{proof}
The proposition boils down to show the following statement: Assume
we have $\omega\in \mathfrak{h}_{(2)}^{n}(X,\V_m)$ and $\alpha \in
A_{X}^{n-1}(\V_m)$ satisfying $D(\alpha)=\omega$, then $\omega=0$.
In order to prove this we write $\omega =\sum_{p+q=n}\omega_{p,q}$
and further
$\omega_{p,q}=\sum_{I\subset\{1,\ldots,n\},|I|=p}\omega_{I,\bar
I^c}$, where $\omega_{I,\bar I^c}=f_{I,\bar J}dz_I\wedge \overline{
dz_{I^c}}$ (see Lemma \ref{fine decomposition of harmonic forms} and
Proposition \ref{pq component of harmonic forms}). It is enough to
show $\omega_{I,\bar I^c}=0$ for all possible $I$. Let $X_{\bar
I^c}$ be the complex manifold considered in the proof of Proposition
\ref{pq component of harmonic forms} for $J=I^c$ and $\bar{X}_{\bar
I^c}$ a smooth toroidal compactification of $X_{\bar I^c}$. Let
$\iota: X\to\bar X$ be the natural inclusion. By Deligne \cite{De0},
the inclusion
$$
\left([\bigoplus_{p+q=\cdot}\sA^{p,0}_{\bar{X}}(\log S)\wedge
\overline{\sA^{q,0}_{\bar{X}}(*S)}]\otimes \bar \V_m,D \right)
\hookrightarrow (\iota_{*}\sA^{\cdot}_{X}(\V_m),D)
$$
is a quasi-isomorphism. Furthermore, by $E_1$-degeneration of the
Hodge filtration, one has also the quasi-isomorphism
$$
\left([\bigoplus_{p+q=\cdot}\sA^{p,0}_{\bar{X}_{\bar I^c}}(\log
S)\wedge \overline{\sA^{q,0}_{\bar{X}_{\bar I^c}}(*S)}]\otimes \bar
\V_m,D\right)\cong
\left([\bigoplus_{p+q=\cdot}\sA^{p,0}_{\bar{X}_{\bar I^c}}(\log
S)\wedge \overline{\sA^{q,0}_{\bar{X}_{\bar I^c}}(*S)}]\otimes \bar
\V_m,D''_{\bar{X}_{\bar I^c}} \right).
$$
It is not difficult to check that $\omega\in
[\bigoplus_{p+q=n}A^{p,0}_{\bar{X}_{\bar I^c}}(\log S)\wedge
\overline{A^{q,0}_{\bar{X}_{\bar I^c}}(*S)}]\otimes \bar \V_m$. So
by the quasi-isomorphisms, we find actually $\alpha'\in
[\bigoplus_{p+q=n-1}A^{p,0}_{\bar{X}_{\bar I^c}}(\log S)\wedge
\overline{A^{q,0}_{\bar{X}_{\bar I^c}}(*S)}]\otimes \bar \V_m$ such
that $D''\alpha'=\omega$. Note for fixed $I$, $\omega_{I,\bar I^c}$
is holomorphic over $\bar{X}_{\bar I^c}$. One has then
\begin{eqnarray*}
  <\omega_{I,\bar J} ,\omega_{I,\bar J}> &=& <(D''\alpha')_{I,\bar J},\omega_{I,\bar J}> \\
    &=& <(\bar \partial_{\tilde X}\alpha')_{I,\bar J},\omega_{I,\bar J}>+<(\theta_{\tilde{X}}\alpha')_{I,\bar J},\omega_{I,\bar J}> \\
    &=& <\bar \partial_{\tilde X}\alpha',\omega_{I,\bar J}>+<\theta_{\tilde{X}}\alpha',\omega_{I,\bar J}> \\
    &=& <\theta_{\tilde{X}}\alpha',\omega_{I,\bar J}> \\
    &=& < \alpha',\theta^{*}_{\tilde{X}}\omega_{I,\bar J}>\\
    &=& < \alpha',0>\\
    &=& 0,
\end{eqnarray*}
and therefore we get $\omega_{I,\bar J}=0$. So $\omega=0$, and the
proof is completed.

\end{proof}
This proposition allows us to obtain an important byproduct of our study of the MHS on the cohomology groups.
Namely, we are able to show the truth of Conjecture \ref{HZ conjecture} in the case of Hilbert modualr varieties with coefficients.
\begin{theorem}\label{isomorphism of Hodge structures}
The natural isomorphism $r_k: H_{(2)}^{k}(X^*,\V_m)\cong IH^{k}(X^*,\V_m), 0 \leq  k \leq n $ is an isomorphism of Hodge structures.
\end{theorem}
\begin{proof}
Since by Corollary \ref{vanishing of intersection cohomology} the above statement is trivial for $k\neq n$, it suffices
to consider the case $k=n$. Theorem 5.4 and Remark 5.5 (i) \cite{HZ3} asserts that the natural map $H_{(2)}^n(X^*, \V_m)\to  H^n(X,\V_m)$,
which is just the composite of the isomorphism $r_k$ (of real vector spaces) with the natural morphism of mixed Hodge structures $IH^{k}(X^*,\V_m)\to H^n(X,\V_m)$,
is actually a morphism of mixed Hodge structures and its image is identified with the lowest weight of the MHS of $H^n(X,\V_m)$.
Now Proposition \ref{injectivity} implies further: \\
i) The morphism $IH^{k}(X^*,\V_m)\to H^n(X,\V_m)$ is injective and therefore
an isomorphism of Hodge structures $IH^{k}(X^*,\V_m)\cong W_{n+|m|}H^n(X,\V_m)$. \\
ii) An isomorphism of Hodge structures
$H_{(2)}^n(X^*, \V_m)\cong W_{n+|m|}H^n(X,\V_m)$. \\
Altogether, since both are identified with the same Hodge structure, the
$L^2$-cohomology and the intersection cohomology are isomorphic as Hodge structures.
\end{proof}

\begin{lemma}\label{lemma on MHS}
Let $(H_{\R},W_{.},F^{.})$ be a MHS with weights $\ge m+k$ and the
following properties:
$$
H_{\C}=F^0=\cdots=F^{m+n}\supsetneq F^{m+n+1}=0,\ 0=W_{m+k}\subset
\cdots\subset W_{2(m+k)}=H_{\R}
$$
for certain $\frac{k+1}{2}\leq n\leq k$. Then the weight filtration
must be of the form
$$
0=W_{m+k}=\cdots=W_{2(m+n)-1}\subsetneq
W_{2(m+n)}=\cdots=W_{2(m+k)}=H_{\R}.
$$
\end{lemma}
\begin{proof}
By the assumption on the Hodge filtration and the Hodge symmetry, it
is easy to see that each graded piece of the weight filtration can
have at most one Hodge type. This implies that the first possible
weight with non-zero dimension is $W_{2(m+n)}$. But then
$W_{2(m+n)}$ must be the whole space. This is because for any $i\geq
2(m+n)+1$, the unique Hodge component $(\frac{i}{2},\frac{i}{2})$ of
$Gr_{i}^{W}$ (assume $i$ even), which is a quotient of
$F^{\frac{i}{2}}\cap W_{i,\C}=0$, is zero. This implies the result.
\end{proof}

\begin{proposition}\label{split MHS}
Let $(H_{\R},W_{.},F^{.})$ be a MHS with weights $\ge m+n$ with
$$
H_{\C}=F^0\supset\cdots\supset F^{m+n}\supsetneq F^{m+n+1}=0,\
0\subset W_{m+n}\subset \cdots\subset W_{2(m+n)}=H_{\R}.
$$ Let $H_{\R}=H_{1,\R}\oplus H_{2,\R}$ be a vector space
decomposition. Assume that $H_{2,\R}\subset F^{m+n}$ and
$H_{1,\R}\subset W_{m+n}$. Then the weight filtration is of the form
$0\subset H_{1,\R}=W_{m+n}=\cdots=W_{2(m+n)-1}\subsetneq
W_{2(m+n)}=H_{\R}$, and the MHS $(H_{\R},W_{.},F^{.})$ is split over
$\R$.
\end{proposition}
\begin{proof}
Consider the quotient MHS on
$(\frac{H_{\R}}{W_{m+n}},\tilde{W}_.,\tilde{F}^{.})$, where $\sim$
means the quotient filtration. By the assumption on $H_{1,\R}$ and
$H_{2,\R}$, one sees that the above quotient MHS is of the form in
Lemma \ref{lemma on MHS}. Thus one obtains the assertion about the
weight filtration $W_.$ except the equality $W_{m+n}=H_{1,\R}$.

Set $H^{p,q}:=F^p\cap \bar F^q\cap W_{p+q,\C}$. As $W_{m+n}$ is of the lowest weight in the weight filtration,
it has a pure Hodge structure of weight $m+n$ induced by $F^.$ and its Hodge
$(p,q)$-component on $W_{m+n}$ is given by $F^p\cap \bar F^q\cap
W_{m+n,\C}$. So $W_{m+n,\C}=\bigoplus_{p+q=m+n}H^{p,q}$ and
$H^{p,q}\cap W_{m+n,\C}=0$ for $p+q\neq m+n$. Because for
$m+n<p+q<2(m+n)$ $W_{p+q}=W_{m+n}$, one has $H^{p,q}=H^{p,q}\cap
W_{m+n,\C}=0$. Now consider the weight $2(m+n)$ pure Hodge structure
on $Gr_{2(m+n)}^{W}$. By Hodge symmetry and the indexing of the
Hodge filtration, $H^{p,2(m+n)-p}$ is zero unless $p=m+n$. In that
case, $H^{m+n,m+n}=F^{m+n}\cap \bar F^{m+n}$. Note that
$H_{2,\R}\subset H^{m+n,m+n}$ by the assumption. This implies the
following relation:
$$
W_{2(m+n),\C}=H_{\C}=H_{1,\C}\oplus H_{2,\C}\subset W_{m+n,\C}\oplus
H^{m+n,m+n}\subset W_{2(m+n),\C}.
$$
Therefore $H_{2,\R}=W_{m+n}$ and $H_{2,\C}=H^{m+n,m+n}$ hold, which
also shows the relation $W_{l,\C}=\bigoplus_{p+q\leq l}H^{p,q}$ for
each $m+n\leq l\leq 2(m+n)$. To show that $(H_{\R},W_{.},F^{.})$ is
split over $\R$, it remains to show that $F^p=\bigoplus_{r\geq
p}H^{r,s}$ holds for each $p$ (see \S2 in \cite{CKS}). Because
$H=W_{m+n,\C}\oplus H_{2,\C}$ as shown above, and $H_{2,\C}\subset
F^{m+n}$ by assumption, it follows that $F^p=F^p(W_{m+n,\C})\oplus
H_{2,\C}$ for each $p$. Now that $F^p(W_{m+n,\C})=\bigoplus_{r\geq
p}H^{r,m+n-r}$, one obtains then $F^p=\bigoplus_{r\geq p}H^{r,s}$
for each $p$. This proves the result.
\end{proof}

Now we proceed to deduce our main results of the paper from the
above results, together with the established information in previous
sections. Let us return to the decomposition (see \S\ref{Eisenstein
cohomology}):
$$
H^k(X,\V_\R)=H_{!}^{k}(X,\V_\R)\oplus H_{\Eis}^{k}(X,\V_\R).
$$
By Proposition \ref{injectivity} and Theorem \ref{isomorphism of Hodge structures}, we denote again by $IH^k(X^*,
\V_m)$ the image of it in $H^k(X,\V_\R)$. The cohomology classes in
$H^k_!(X,\V_\R)$ are representable by differential forms with
compact support, which are square integrable with respect to
\emph{any} complete K\"{a}hler metric on $X$. Therefore
$H^k_!(X,\V_\R)\subset H_{(2)}^k(X^*, \V_m)=IH^k(X^*,
\V_m)$. The following result
improves Theorem \ref{Li-Schwermer}:
\begin{theorem}\label{improvement of Li-Schwemer}
For $k\neq n$, one has $H^k(X,\V_m)=H_{\Eis}^k(X,\V_m)$.
Furthermore, for $0\leq k\leq n-1$ and $k=2n$, it holds
$H^k(X,\V_m)=0$, and for $n+1\leq k\leq 2n-1$,
$H^k(X,\V_m)=H_{\Eis}^k(X,\V_m)\stackrel{r}{\cong} H^k(\partial
X^\sharp,\V_m)$.
\end{theorem}
\begin{proof}
By the above discussion, one knows that $$\dim
H_{\Eis}^k(X,\V_m)\leq \dim H^k(X,\V_m)\leq \dim H_{(2)}^k(X^*,
\V_m)+\dim H_{\Eis}^k(X,\V_m).$$ By Corollary \ref{vanishing of
intersection cohomology}, it follows that
$H^k(X,\V_m)=H_{\Eis}^k(X,\V_m)$ for $k\neq n$. The remaining part
of the theorem follows from Lemma \ref{lemma of Zucker} and Theorem
\ref{Li-Schwermer}. Also one notices that $H_{\Eis}^{2n}(X,\V_m)=0$,
since $\partial X^\sharp$ is of real dimension $2n-1$.
\end{proof}

\begin{theorem}\label{mixed hodge structure of cohomology}
Let $\V_m$ be the irreducible non-trivial local system as above and
the cohomology group $H^k(X,\V_m)$ is equipped with Saito's MHS. Then:
\begin{itemize}
\item [(i)] For $n+1\leq k\leq 2n-1$, one has $H^k(X,\V_m)=H^k_{\Eis}(X,\V_m)$ and the MHS on
$H^k(X,\V_m)$ is pure and of pure type $(|m|+n,|m|+n)$.
\item [(ii)] $IH^n(X^*, \V_m)=H^{n}_{!}(X,\V_{m})$ and
$H^n(X,\V_m)=IH^n(X^*, \V_m)\oplus H^n_{\Eis}(X,\V_m)$ is the
splitting of the weight filtration over $\R$.
\item [(iii)]  For $n \leq k\leq 2n-1$ the weight filtration is of the following form:
$$
0\subset W_{|m|+k}=\cdots=W_{2(|m|+n)-1}\subset
W_{2(|m|+n)}=\cdots=W_{2(|m|+k)}=H^k(X,\V_m),
$$
where $W_{|m|+k}=IH^k(X^*, \V_m)$ and $Gr_{2(|m|+n)}^{W}\cong
H^k_{\Eis}(X,\V_m)$. The Hodge filtration is of the following form:
$$
H^k(X,\V_m)=F^0\supset \cdots \supset F^{|m|+n}\supset 0,
$$
in which $H^k_{\Eis}(X,\V_m)\subset F^{|m|+n}$ holds.
\end{itemize}
\end{theorem}

\begin{proof}
For $n+1\leq k\leq 2n-1$, (i) and (iii) follows from Theorem
\ref{improvement of Li-Schwemer}, \ref{type of Eisenstein coho} and
Lemma \ref{lemma on MHS}. For $k=n$, one applies Proposition
\ref{split MHS} for $H_{\R}=H^n(X,\V_m)$, $H_{1,\R}=H^n_{!}(X,\V_m)$
and $H_{2,\R}=H^n_{\Eis}(X,\V_m)$. The condition for $H_2$ follows
from Theorem \ref{type of Eisenstein coho}. The relation
$H_{1,\R}\subset IH^n(X^*, \V_m)\subset W_{|m|+n}$ follows from
the above discussion. Then Proposition \ref{split MHS} implies that
$$
H_{1,\R}=IH^n(X^*, \V_m)= W_{|m|+n},
$$
and the splitting of the MHS over $\R$.
\end{proof}
For the MHS $(H^k(X,\V_m),W_.,F^.)$, put
$$
h^{P,Q}_k:=\dim Gr_{F}^{P}Gr_{\bar F}^{Q}Gr_{P+Q}^{W}H^k(X,\V_m),\
H^{P,Q}_k:=F^{P}\cap \bar F^{Q}\cap W_{P+Q,\C}.$$ By Theorem
\ref{mixed hodge structure of cohomology}, $\dim
H^{P,Q}_k=h^{P,Q}_k$.

\begin{theorem}\label{dimension formula of other degrees}
Notation as above. The following statements hold:
\begin{itemize}
\item[(i)] $H^{k}(X,\V_m)=0$ for $0\leq k\leq n-1$ and $k=2n$.
\item[(ii)] If $m_1=\cdots=m_n$, then for $n+1\leq k\leq 2n-1$
$$
h_k^{|m|+n,|m|+n}:=\dim_\C
F^{|m|+n}W_{2(|m|+n)}H^{k}(X,\V_m)=\dim_\C
H^{k}(X,\V_m)=\binom{n-1}{k-n}h,
$$
where $h$ is the number of cusps.
\item[(iii)] If not all $m_i$ are equal, then $H^{k}(X,\V_m)=0$ for $n+1\leq k\leq 2n-1$.
\end{itemize}
\end{theorem}

\begin{proof}
They follow directly from Theorem \ref{improvement of Li-Schwemer},
\ref{mixed hodge structure of cohomology}, \ref{type of Eisenstein
coho} and Lemma \ref{lemma of Zucker}.
\end{proof}

\begin{theorem}\label{algebraic description of Hodge components}
One has the following natural isomorphisms:
\begin{itemize}
\item [(i)] For $n+1\leq k\leq 2n-1$,
$H_k^{|m|+n,|m|+n}\cong H^{k-n}(\bar X,
\bigotimes_{i=1}^{n}\sL^{m_i+2}_{i})$.
  \item [(ii)]
  $H_n^{|m|+n,0}\cong H^{0}(\bar X, \sO_{\bar X}(-S)\otimes
\bigotimes_{i=1}^{n}\sL^{m_i+2}_{i}),\ H_n^{|m|+n,|m|+n}\cong
H^{0}(S, \bigotimes_{i=1}^{n}\sL^{m_i+2}_{i}|_S)$, and for $0\leq
P\leq |m|+n-1, \ P+Q=|m|+n$,
$$
H_n^{P,Q}\cong \bigoplus_{\stackrel{I\subset \{1,\ldots,n\},}{
|m_I|+|I|=P}}H^{k-|I|}(\bar X, \bigotimes_{i\in
I}\sL^{m_i+2}_{i}\otimes \bigotimes_{i\in I^c}\sL^{-m_i}_{i}).
$$
\end{itemize}
\end{theorem}
\begin{proof}
(i) follows directly from Theorem \ref{mixed hodge structure of
cohomology} (i) and Corollary \ref{algebraic description}. By
Theorem \ref{mixed hodge structure of cohomology}, one has for
$0\leq P\leq |m|+n-1$ and $P+Q=|m|+n$,
$$
H_n^{P,Q}=Gr_{F}^{P}H^n(X,\V_m).
$$
The isomorphisms for these $H_n^{P,Q}$ follow from Corollary
\ref{algebraic description}. By Theorems \ref{mixed hodge structure
of cohomology} and \ref{dimension formula of other degrees} (ii),
one has
$$
Gr_{F}^{|m|+n}H^n(X,\V_m)=F^{|m|+n}=H_n^{|m|+n,0}\oplus
H_n^{|m|+n,|m|+n},
$$
and $H_n^{|m|+n,0}=H_{(2)}^{|m|+n,0}(X^*,\V_m)$. By Corollary
\ref{hodge decomp of inter cohomlogy}, $H_n^{|m|+n,0}\cong
H^{0}(\bar X, \sO_{\bar X}(-S)\otimes
\bigotimes_{i=1}^{n}\sL^{m_i+2}_{i})$, and by Corollary
\ref{algebraic description}, $Gr_{F}^{|m|+n}H^n(X,\V_m)\cong
H^{0}(\bar X, \bigotimes_{i=1}^{n}\sL^{m_i+2}_{i})$ (the previous
two isomorphisms are in fact equalities). Finally, one notes that
$H^{1}(\bar X, \sO_{\bar X}(-S)\otimes
\bigotimes_{i=1}^{n}\sL^{m_i+2}_{i})$ is dual to $H^{n-1}(\bar
X,\bigotimes_{i=1}^{n}\sL^{-m_i}_{i})$ by the Serre duality. Since
it is zero by the next lemma, the long exact sequence of sheaf
cohomologies of the following short exact sequence
$$
0\to \sO_{\bar X}(-S)\otimes \bigotimes_{i=1}^{n}\sL^{m_i+2}_{i}\to
\bigotimes_{i=1}^{n}\sL^{m_i+2}_{i} \to
\bigotimes_{i=1}^{n}\sL^{m_i+2}_{i}|_S\to 0
$$
yields the isomorphism $H_n^{|m|+n,|m|+n}\cong
H^0(S,\bigotimes_{i=1}^{n}\sL^{m_i+2}_{i}|_S)$.
\end{proof}

\begin{lemma}\label{euler
characteristic} For $l<n$, $H^l(\bar X,
\bigotimes_{i=1}^{n}\sL_i^{-m_i})=0$.
\end{lemma}
\begin{proof}
By Propositions \ref{degeneration of ss} and \ref{cohomology sheaves
matrix}, $\dim H^{l}(\bar X, \bigotimes_{i=1}^{n}\sL_i^{-m_i})\leq
\dim H^{l}(X,\V_m)$, which is zero for $0\leq l\leq n-1$ by Theorem
\ref{improvement of Li-Schwemer}. Thus $H^{l}(\bar X,
\bigotimes_{i=1}^{n}\sL_i^{-m_i})=0$ for $l<n$.
\end{proof}
We end this paper with some discussions on
$\dim H^0(\bar{X}_{\bar{J}}, \sO_{\bar{X}_{\bar{J}}}(-S)\otimes
\bigotimes_{j=1}^{n}\sL_{j}^{m_j+2})$, as well as the dimension
of the middle degree cohomology.
\begin{proposition}\label{Riemann-Roch formula}
Let $c_i$ be the $i$-th chern class of $\bar X_{\bar J}$, $c'_i$ be
the $i$-th chern class of $T_{\bar{X}_{\bar{J}}}(-\log S_{\bar J})$,
the dual vector bundle of $\Omega^1_{\bar{X}_{\bar J}}(\log S_{\bar
J})$, and $P(c_1,\ldots,c_n)$ be the degree $n$ polynomial computing
$\chi(\bar{X}_{\bar J},\sO_{\bar{X}_{\bar J}})$ in the
Hirzebruch-Riemann-Roch formula. The following formula holds:
$$
\dim H^0(\bar{X}_{\bar{J}}, \sO_{\bar{X}_{\bar{J}}}(-S)\otimes
\bigotimes_{j=1}^{n}\sL_{j}^{m_j+2})=(-1)^n[\prod_{i=1}^{n}(m_i+1)-1]P(c'_1,\ldots,c'_n)+(-1)^nP(c_1,\ldots,c_n).
$$
\end{proposition}
\begin{proof}
Consider first the $J=\emptyset$ case. The above lemma implies that
\begin{eqnarray*}
   \dim H^0(\bar X, \sO_{\bar X}(-S)\otimes
\bigotimes_{j=1}^{n}\sL_{j}^{m_j+2})&=&  \dim H^n(\bar X, \bigotimes_{i=1}^{n}\sL_i^{-m_i})\\
&=&(-1)^n\chi(\bar X, \bigotimes_{i=1}^{n}\sL_i^{-m_i})  \\
   &=&(-1)^n[\prod_{i=1}^{n}(m_i+1)-1]\prod_{i=1}^{n}\sL_{i}+(-1)^n\chi(\bar X,\sO_{\bar
X}).
\end{eqnarray*}
By Hirzebruch proportionality in the non-compact case (see Theorem
3.2 in \cite{Mum}) and Proposition \ref{basic logarithmic higgs
bundles}, it follows that
$\prod_{i=1}^{n}\sL_{i}=P(c'_1,\ldots,c'_n)$. For a general $J$, one
notices that the corresponding statement in Lemma \ref{euler
characteristic} holds also for $X_{\bar J}$. The above argument
works verbatim for the general case.
\end{proof}
However one can not conclude that
$P(c_1,\ldots,c_n)=P(c'_1,\ldots,c'_n)$ in general. In fact, by
Proposition \ref{L2 condition is algebraic}, the space
$H^0(\bar{X}_{\bar{J}}, \sO_{\bar{X}_{\bar{J}}}(-S)\otimes
\bigotimes_{j=1}^{n}\sL_{j}^{m_j+2})$ is exactly the space of cusp
forms on $\HH^n$ with respect to the discrete subgroup $\Gamma_{J}$.
The dimension formula of cusp forms by means of Selberg's trace
formula in such a case is the main result of the work \cite{Shimizu}
by Shimizu for regular $m$ (see \cite{Ish} for the irregular case).
By his formula in the regular case (Theorem 11 \cite{Shimizu}) the
dimension is not proportional to $\prod_{i=1}^{n}(m_i+1)$ because of
an error term coming from the cusps. The same formula also shows
that the dimension of cusp forms with respect to $\Gamma_J$ is
generally different from that with respect to $\Gamma$. These
subtleties do not arise in the case of compact quotients studied by
Matsushima-Shimura \cite{MS}. Now for each subset $J\subset
\{1,\cdots,n\}$, we put
$$
h(J,m)=\dim H^0(\bar{X}_{\bar{J}},
\sO_{\bar{X}_{\bar{J}}}(-S)\otimes
\bigotimes_{j=1}^{n}\sL_{j}^{m_j+2}).
$$
It is independent of the choice of a smooth toroidal
compactification of $X_{\bar J}$ and $h(J,m)=h(J^c,m)$ by Hodge
isometry.
\begin{proposition}\label{dimension formula of middle degree}
Let $\V_m$ be an irreducible non-trivial local system over $X$. Put
$\delta(m)=1$ when $m_1=\cdots=m_n$ is satisfied and otherwise zero.
Then it holds that
$$
\dim H^{n}(X,\V_m)=\delta(m)h+\sum_{I\subset \{1,\cdots,n\}}h(I,m).
$$
Moreover, for $P+Q=|m|+n$,
$$
h^{P,Q}_{n}=\sum_{\stackrel{I\subset \{1,\ldots,n\},}{
|m_I|+|I|=P}}h(I,m),
$$
and $h_n^{|m|+n,|m|+n}=\delta(m)h$. Otherwise $h_n^{P,Q}=0$.
\end{proposition}
\begin{proof}
It follows from Corollary \ref{hodge decomp of inter cohomlogy},
Propositions \ref{byproduct}, \ref{L2 condition is algebraic},
Theorem \ref{mixed hodge structure of cohomology} (ii), \ref{type of
Eisenstein coho} and Lemma \ref{lemma of Zucker}.
\end{proof}

\end{document}